\documentclass[final,1pt,times]{elsarticle}

\biboptions{sort&compress}

\usepackage{amsmath}
\usepackage{mathrsfs}  
\usepackage{amssymb}
\usepackage{amsthm}
\usepackage{bm}
\usepackage{layouts}

\graphicspath{{./figs/}}
\usepackage[pdftex,colorlinks=true,bookmarks=true,citecolor=blue,urlcolor=blue]{hyperref} 
\usepackage[caption=false]{subfig}

\newcommand{\field}[1]{\mathbb{#1}}
\newcommand{\fs}[1]{\mathsf{#1}}

\DeclareMathOperator*{\supp}{supp}

\DeclareMathOperator{\diag}{diag}
\DeclareMathOperator{\Wrons}{\mathscr{W}}

\DeclareMathOperator{\trace}{tr}
\newcommand{\tp}{\intercal}

\newcommand{\ovl}[1]{\overline{#1}}

\newcommand{\bigO}[1]{\mathop{\mathscr{O}}\left(#1\right)}

\let\Re\relax
\DeclareMathOperator{\Re}{Re}
\let\Im\relax
\DeclareMathOperator{\Im}{Im}
\newcommand{\vv}[1]{\boldsymbol{#1}}
\newcommand{\vs}[1]{\boldsymbol{#1}}

\DeclareMathOperator{\fourier}{\mathscr{F}}

\newcommand{\wtilde}[1]{\widetilde{#1}}

\newcommand{\et}{\textit{et~al.}}

\newtheorem{theorem}{Theorem}[section]
\newtheorem{prop}[theorem]{Proposition}

\newtheorem{lemma}[theorem]{Lemma}

\newtheorem{rem}{Remark}[section]

\usepackage{enumitem}

\journal{Communications in Nonlinear Science and Numerical Simulation}
\begin{document}
\begin{frontmatter}

\title{Numerical Methods for Fast Nonlinear Fourier Transformation, Part I: Exponential 
Runge-Kutta and Linear Multistep Methods}
\author{Vishal Vaibhav}
\ead{vishal.vaibhav@gmail.com}
\address{Delft Center for Systems and Control, 
Delft University of Technology, Mekelweg 2. 2628 CD Delft, 
The Netherlands}

\begin{abstract}
The main objective of this series of papers is to explore the entire landscape
of numerical methods for fast nonlinear Fourier transformation (NFT) within the class of integrators
known as the exponential integrators. In this paper, we explore the theoretical
aspects of exponential Runge-Kutta (RK) and linear multistep (LM) methods, in
particular, the stability and convergence of these methods via the transfer 
matrix formulation. The analysis carried out in the paper shows that while the
exponential LM methods are naturally amenable to FFT-based fast polynomial
arithmetic, the RK methods require equispaced nodes to achieve that. Therefore,
each these family of methods is capable of yielding a family of fast NFT algorithms 
such that the scattering coefficients can be computed with a complexity of 
$\bigO{N\log^2N}$ and a rate of convergence given by $\bigO{N^{-p}}$ where $N$ 
is the number of samples of the signal and $p$ is order of the underlying
discretization scheme. Further,
while RK methods can accommodate vanishing as well as periodic boundary
conditions, the LM methods can only handle the former type of boundary
conditions without requiring a starting procedure. The ideas presented in 
this paper extend naturally to the family of integrators known as general 
linear methods which will be explored in a forthcoming paper. 
\end{abstract}

\end{frontmatter}

\section*{Notations}
\label{sec:notations}
The set of real numbers (integers) is denoted by $\field{R}$ ($\field{Z}$) and 
the set of non-zero positive real numbers (integers) by $\field{R}_+$ 
($\field{Z}_+$). The set of complex numbers are denoted by $\field{C}$,
and, for $\zeta\in\field{C}$, $\Re(\zeta)$ and $\Im(\zeta)$ refer to the real
and the imaginary parts of $\zeta$, respectively. The complex conjugate of 
$\zeta\in\field{C}$ is denoted by $\zeta^*$ and $\sqrt{\zeta}$ denotes its
square root with a positive real part. The upper-half (lower-half) of $\field{C}$ 
is denoted by $\field{C}_+$ ($\field{C}_-$) and it closure by $\ovl{\field{C}}_+$
($\ovl{\field{C}}_-$). The set $\field{D}=\{z|\,z\in\field{C},\,|z|<1\}$
denotes an open unit disk in $\field{C}$ and $\ovl{\field{D}}$ denotes its 
closure. The set $\field{T}=\{z|\,z\in\field{C},\,|z|=1\}$ denotes the unit 
circle in $\field{C}$. The Pauli's spin matrices are denoted by, 
$\sigma_j,\,j=1,2,3$, which are defined as
\[
\sigma_1=\begin{pmatrix}
0 &  1\\
1 &  0
\end{pmatrix},\quad
\sigma_2=\begin{pmatrix}
0 &  -i\\
i &  0
\end{pmatrix},\quad
\sigma_3=\begin{pmatrix}
1 &  0\\
0 & -1
\end{pmatrix},
\]
where $i=\sqrt{-1}$. For uniformity of notations, we denote
$\sigma_0=\text{diag}(1,1)$. Matrix transposition is denoted by $(\cdot)^\tp$
and $I$ denotes the identity matrix. For any two vectors 
$\vv{u},\vv{v}\in\field{C}^2$, $\Wrons(\vv{u},\vv{v})\equiv (u_1v_2-u_2v_1)$ 
denotes the Wronskian of the two vectors. The kronecker product of two 
matrices $A$ and $B$ is denoted by $A\otimes B$. The support of a function
$f:\Omega\rightarrow\field{R}$ in $\Omega$ is defined as $\supp
f=\ovl{\{x\in\Omega|\,f(x)\neq0\}}$. The Lebesgue spaces of complex-valued 
functions defined in $\field{R}$ are denoted by 
$\fs{L}^p$ for $1\leq p\leq\infty$ with their corresponding 
norm denoted by $\|\cdot\|_{\fs{L}^p}$ or $\|\cdot\|_p$. The class of continuous 
functions is denoted by $\fs{C}$ and $\fs{C}^{m}$ comprises function that are at
least $m$-times differentiable.
The inverse Fourier-Laplace transform of a function $F(\zeta)$ analytic in 
$\ovl{\field{C}}_+$ is defined as
\[
\vv{f}(\tau)=\frac{1}{2\pi}\int_{\Gamma}F(\zeta)e^{i\zeta\tau}\,d\zeta,
\]
where $\Gamma$ is any contour parallel to the real line.

\section{Introduction}
The main objective of this series of papers is to provide a comprehensive survey
of the numerical methods available for designing nonlinear Fourier transform
(NFT) algorithms with desired stability properties, error behavior and (low) 
complexity of computation. Our primary interest is in the class of 
integrators known as \emph{exponential integrators} which were recently shown to 
be promising in designing fast NFT algorithms~\cite{V2017INFT1,V2018LPT} that 
are stable as well as capable of
achieving higher orders of convergence. It must be noted that the family of methods 
considered in the aforementioned works hardly represent the totality of the methods 
available for the numerical solution of ordinary differential equations (ODEs).
Besides, there are several open problems still to be addressed, for instance, the work presented 
in~\cite{V2018LPT} does not address the NFTs with periodic boundary condition,
or, beyond the second order of convergence, the differential approach has not
been shown to be successful in devising inverse NFT 
algorithms~\cite{BLK1985,V2017INFT1,V2018BL, V2018TL}. Therefore, 
the purpose of this series of papers is to explore the entire
landscape of available numerical methods that can yield low complexity
algorithms for direct NFT (and perhaps an inverse NFT within the differential
approach), and, characterize their utility 
under various constraints and performance requirements. In this paper, we consider two
families of integrators known as the \emph{Runge-Kutta} method and \emph{linear multistep
methods}~\et~\cite{HNW1993}.  

Our exposition begins by collecting some of the basic theoretical results 
that characterize the signal processing aspects of NFTs. The main
results of this paper are contained in Sec.~\ref{eq:num-scheme} where we present 
the numerical scheme. It is quite fortunate that the existing
literature on the classical methods considered in this work for the solution of
the Zakharov-Shabat problem (introduced later in this section) is entirely 
sufficient for our purpose. In this light, the primary contribution of this paper consists in unifying
the approach for fast NFTs and demonstrating its connection with some of the
time tested methods for ODEs. These areas of research, however, are still very alive given
that the quest for designing efficient Runge-Kutta and related methods is far
from complete. For instance, general linear methods which appeared in the work of 
J.~C. Butcher~\cite{Butcher2003,J2009} is an evolving novel technique that essentially
combines the benefits of Runge-Kutta and linear multistep methods--this would be
the subject matter of a forthcoming paper in this series.

\subsection{The AKNS system}
\label{sec:akns}
We begin our discussion with a brief review of the scattering theory closely
following the formalism presented in~\cite{AKNS1974}. The nonlinear Fourier
transform of any signal is defined via the Zakharov-Shabat (ZS)~\cite{ZS1972} scattering 
problem which can be stated as follows: For $\zeta\in\field{R}$ and $\vv{v}=(v_1,v_2)^{\tp}$,
\begin{equation}\label{eq:ZS-prob}
\vv{v}_t = -i\zeta\sigma_3\vv{v}+U(t)\vv{v},
\end{equation}
where $\sigma_3$ is one of the Pauli matrices defined in the beginning of this
article. The potential 
$U(t)$ is defined by $U_{11}=U_{22}=0,\,U_{12}=q(t)$ and $U_{21}=r(t)$ with $r=\kappa q^*$
($\kappa\in\{+1, -1\}$). Here, $\zeta\in\field{R}$ is known as the \emph{spectral parameter}
and $q(t)$ is the complex-valued signal. The solution of the scattering 
problem~\eqref{eq:ZS-prob}, henceforth referred
to as the ZS problem, consists in finding the so called 
\emph{scattering coefficients} which are defined through special solutions 
of~\eqref{eq:ZS-prob} known as the \emph{Jost solutions} which are linearly 
independent solutions of~\eqref{eq:ZS-prob} such that they have a plane-wave 
like behavior at $+\infty$ or $-\infty$.
\begin{itemize}
\item\emph{First kind}: The Jost solutions of the first kind, denoted
by $\vs{\psi}(t;\zeta)$ and $\ovl{\vs{\psi}}(t;\zeta)$, are the linearly
independent solutions of~\eqref{eq:ZS-prob} which have the following asymptotic 
behavior as $t\rightarrow\infty$:
$\vs{\psi}(t;\zeta)e^{-i\zeta t}\rightarrow(0,1)^{\tp}$ and 
$\ovl{\vs{\psi}}(t;\zeta)e^{i\zeta t}\rightarrow (1,0)^{\tp}$.

\item\emph{Second kind}: The Jost solutions of the second kind, denoted by
$\vs{\phi}(t;\zeta)$ and $\overline{\vs{\phi}}(t;\zeta)$, are the linearly 
independent solutions of~\eqref{eq:ZS-prob} which have the following asymptotic 
behavior as $t\rightarrow-\infty$: 
$\vs{\phi}(t;\zeta)e^{i\zeta t}\rightarrow(1,0)^{\tp}$ and 
$\overline{\vs{\phi}}(t;\zeta)e^{-i\zeta t}\rightarrow(0,-1)^{\tp}$.
\end{itemize}
On account of the linear independence of 
$\vs{\psi}$ and $\ovl{\vs{\psi}}$, we have
\begin{equation*}
\begin{split}
\vs{\phi}(t;\zeta)&=a(\zeta)\ovl{\vs{\psi}}(t;\zeta)
+b(\zeta)\vs{\psi}(t;\zeta),\\
\ovl{\vs{\phi}}(t;\zeta)&=-\ovl{a}(\zeta)\vs{\psi}(t;\zeta)
+\ovl{b}(\zeta)\ovl{\vs{\psi}}(t;\zeta).
\end{split}
\end{equation*}
Similarly, using the pair $\vs{\phi}$ and $\ovl{\vs{\phi}}$, we have
\begin{equation*}
\begin{split}
\vs{\psi}(t;\zeta)&=-a(\zeta)\overline{\vs{\phi}}(t;\zeta)
+\ovl{b}(\zeta)\vs{\phi}(t;\zeta),\\
\ovl{\vs{\psi}}(t;\zeta)&=\overline{a}(\zeta)\vs{\phi}(t;\zeta)
+{b}(\zeta)\ovl{\vs{\phi}}(t;\zeta).
\end{split}
\end{equation*}
The coefficients appearing in the equations above can be written in terms of the
Jost solutions by using the Wronskian relations: 
\begin{equation}\label{eq:wrons-scoeff}
\begin{split}
&a(\zeta)= \Wrons\left(\vs{\phi},{\vs{\psi}}\right),\quad 
 b(\zeta)= \Wrons\left(\ovl{\vs{\psi}},\vs{\phi}\right),\\
&\ovl{a}(\zeta)=\Wrons\left(\ovl{\vs{\phi}},\overline{\vs{\psi}}\right),\quad
\ovl{b}(\zeta) =\Wrons\left(\ovl{\vs{\phi}},{\vs{\psi}}\right).
\end{split}
\end{equation}
These coefficients are known as the \emph{scattering coefficients} and the process
of computing them is referred to as \emph{direct scattering}.

The analytic continuation of the Jost solution with respect to $\zeta$ is
possible provided the potential is decays sufficiently fast or has a compact
support. If the potential has a compact support, the Jost solutions have analytic 
continuation in the entire complex plane. Consequently, the scattering 
coefficients $a(\zeta)$, $b(\zeta)$, $\ovl{a}(\zeta)$, $\ovl{b}(\zeta)$ are 
analytic functions of $\zeta\in\field{C}$~\cite{AKNS1974}. Further, it was shown
in~\cite{V2018CNSNS,V2018TL} that the scattering coefficients are entire functions of
\emph{exponential type}:
\begin{theorem}\label{thm:ab-estimate0}
Let $q\in\fs{L}^{1}$ with support in $\Omega=[-T,T]$ ($T>0$) and set
$\mu=\|q\|_{\fs{L}^1(\Omega)}$. Then the estimates
\begin{align}
&|b(\zeta)|\leq
\sinh(\mu)\times
\begin{cases}
    e^{2T\Im{\zeta}},&\zeta\in\ovl{\field{C}}_+,\\
    e^{-2T\Im{\zeta}},&\zeta\in{\field{C}}_-,
\end{cases}\label{eq:b-estimate-result}\\
&|\tilde{a}(\zeta)|\leq
[\cosh(\mu)-1]\times
\begin{cases}
    e^{2T\Im{\zeta}},&\zeta\in\ovl{\field{C}}_+,\\
    e^{-2T\Im{\zeta}},&\zeta\in{\field{C}}_-,
\end{cases}\label{eq:a-estimate-result}
\end{align}
where $\tilde{a}(\zeta)$ denotes $[a(\zeta)-1]e^{-2i\zeta T}$, hold.
\end{theorem}
An immediate consequence of this theorem together with the Paley-Wiener 
theorem~\cite[Chap.~VI]{Yosida1968} is that 
\begin{equation}
\supp\fourier^{-1}[f]\subset[-2T,2T],
\end{equation}
where $f(\zeta)$ denotes either $b(\zeta)$ or $\tilde{a}(\zeta)$. The Fourier transformation 
in the above equation is understood in the sense of distributions.
 
Let $q(t)\in\fs{L}^1$ such that $|q(t)|\leq C\exp[-2\delta|t|]$ almost everywhere in
$\field{R}$ for some constants $C>0$ and $\delta>0$. Then, the functions 
$e^{-i\zeta t}\vs{\psi}$ and $e^{i\zeta t}\vs{\phi}$ 
are analytic in the half-space $\{\zeta\in\field{C}|\,\Im{\zeta}>-\delta\}$. The functions 
$e^{i\zeta t}\overline{\vs{\psi}}$ and $e^{-i\zeta t}\ovl{\vs{\phi}}$ 
are analytic in the half-space $\{\zeta\in\field{C}|\,\Im{\zeta}<\delta\}$. In this 
case, the coefficient $a(\zeta)$ is analytic for $\Im{\zeta}>-\delta$ while the coefficient 
$b(\zeta)$ is analytic in the strip defined by $-\delta<\Im\zeta<\delta$. Further
characterization of time-limited and one-sided signals can be found
in~\cite{V2018TL}. 

Finally, assuming that the analytic continuation is guaranteed, the restriction $r=\kappa q^*$ yields
\begin{align}
&\ovl{\vs{\psi}}(t;\zeta)=
\begin{pmatrix}
\psi_2^*(t;\zeta^*)\\
\kappa\psi_1^*(t;\zeta^*)\\
\end{pmatrix}
=\begin{cases}
\sigma_1\,\vs{\psi}^*(t;\zeta^*),&\kappa=+1,\\
i\sigma_2\,\vs{\psi}^*(t;\zeta^*),&\kappa=-1,
\end{cases}\\
&\ovl{\vs{\phi}}(t;\zeta)
=-\begin{pmatrix}
\kappa\phi_2^*(t;\zeta^*)\\
\phi_1^*(t;\zeta^*)\\
\end{pmatrix}
=\begin{cases}
-\sigma_1\,\vs{\phi}^*(t;\zeta^*),&\kappa=+1,\\
i\sigma_2\,\vs{\phi}^*(t;\zeta^*),&\kappa=-1;
\end{cases}
\end{align}
consequently,
\begin{equation}
\ovl{a}(\zeta)=a^*(\zeta^*),\quad\ovl{b}(\zeta)=-\kappa b^*(\zeta^*).
\end{equation}

The scattering coefficients together with certain 
quantities defined below that facilitate the recovery of the scattering potential are 
collectively referred to as the \emph{scattering data}. The \emph{nonlinear Fourier
spectrum} can then be defined as any of the subsets which qualify as the ``primordial''
scattering data~\cite[App.~5]{AKNS1974}, i.e., the minimal set of quantities 
sufficient to determine the scattering potential, uniquely.

In general, the nonlinear Fourier spectrum for the potential $q(t)$ comprises 
a \emph{discrete} and a \emph{continuous spectrum}. The discrete spectrum consists 
of the so called \emph{eigenvalues} $\zeta_k\in\field{C}_+$, such that 
$a(\zeta_k)=0$, and, the \emph{norming constants} $b_k$ such that 
$\vs{\phi}(t;\zeta_k)=b_k\vs{\psi}(t;\zeta_k)$. For convenience, let the
discrete spectrum be denoted by the set
\begin{equation}\label{eq:set-discrete-spectrum}
\mathfrak{S}_K=\{(\zeta_k,b_k)\in\field{C}^2|\,\Im{\zeta_k}>0,\,k=1,2,\ldots,K\}.
\end{equation}
For compactly supported potentials, $b_k=b(\zeta_k)$. Note that some authors choose to
define the discrete spectrum using the pair $(\zeta_k,\rho_k)$ where
$\rho_k=b_k/\dot{a}(\zeta_k)$ is known as the \emph{spectral amplitude} corresponding
to $\zeta_k$ ($\dot{a}$ denotes the derivative of $a$). The continuous spectrum, also 
referred to as the \emph{reflection coefficient}, is 
defined by $\rho(\xi)={b(\xi)}/{a(\xi)}$ for $\xi\in\field{R}$. 

In the following, we collect some of the useful results on the characterization
of signals synthesized using the inverse NFT when starting 
from the reflection coefficient where the discrete spectrum is assumed to be empty
unless otherwise stated. These results are
extremely important in that their characterization remains invariant under
evolution consistent with the AKNS-class of integrable nonlinear equations. An interesting
result which is due to Van der Mee~\cite{C2014,C2015} which prescribes the
conditions under which $q$ is square integrable is as follows:
\begin{theorem}[Square integrable signals~\cite{C2014,C2015}]\label{thm:L2}
Let $\rho(\xi)\in\fs{L}^2\cap\fs{L}^{\infty}$. For $\kappa=+1$, it is additionally assumed 
that $|\rho|<1$. Then, there exists a unique scattering potential $q\in\fs{L}^2$.
\end{theorem}
Another important class of signals we would like to consider are 
such that their discrete spectrum is empty and the continuous spectrum is 
compactly supported. In analogy with bandlimited
signals in conventional Fourier analysis, such signals are referred to as
\emph{nonlinearly bandlimited}. Such signals lend themselves well to the
development of fast inverse NFT algorithm~\cite{V2018BL}. The characterization of the such
signals, in part, is contained in the preceding theorem as a special case. A constructive 
approach based on sampling expansion presented in~\cite{V2018BL1} yields
somewhat stronger result:
\begin{theorem}[Nonlinearly bandlimited signals~\cite{V2018BL1}]\label{thm:NBL}
Let $\rho(\xi)\in\fs{L}^{\infty}$ with $\supp
\rho\subset[-\sigma,\sigma]$. For $\kappa=+1$, it is additionally assumed that $|\rho|<1$.
Then, there exists a unique scattering potential $q\in\fs{L}^2\cap\fs{C}^{\infty}$.
\end{theorem}
\begin{proof}
A constructive proof of this theorem can be obtained using the approach presented in~\cite{V2018BL1} 
where it is shown that the solution of the Gelfand-Levitan-Marchenko equation exists
under the conditioned prescribed in the statement of theorem and $q\in\fs{L}^2$ by 
construction. The expression obtained for $q$ is manifestly differentiable to
any order.
\end{proof}
The work~\cite{V2018BL1} also contains several algorithms for synthesizing the 
aforementioned class of signals with extremely high accuracy.

\subsection{Periodic Boundary Conditions}
We would like to qualify our numerical methods for solution of the ZS problem
when the potential is periodic; therefore, we briefly introduce the scattering theory 
for periodic potentials. In this section, we assume that $q(t)$ is 
a complex-valued periodic signal whose period is $2T$. 
The solution of the scattering problem proceeds by computing two linearly independent
solutions
\begin{equation}
\vs{\phi}(t;\zeta)=
\begin{pmatrix}
\phi_1(t;\zeta)\\
\phi_2 (t;\zeta)
\end{pmatrix},\quad
\ovl{\vs{\phi}}(t;\zeta)=
\begin{pmatrix}
\phi^*_2(t;\zeta^*)\\
\kappa\phi^*_1 (t;\zeta^*)
\end{pmatrix},
\end{equation}
such that 
\begin{equation}
\vs{\phi}(t_0;\zeta)=
\begin{pmatrix}
1\\
0
\end{pmatrix}.
\end{equation}
On account of the periodicity of the potential, 
$\vs{\phi}(t+2T;\zeta)$ is also a solution so that it can be expanded in terms of the
linearly independent solutions, $\vs{\phi}$ and $\ovl{\vs{\phi}}$, as 
\begin{equation}
\vs{\phi}(t+2T)=
a(\zeta;t_0)\vs{\phi}(t;\zeta)+
b(\zeta;t_0)\ovl{\vs{\phi}}(t;\zeta),
\end{equation}
where $a(\zeta;t_0)$ and $b(\zeta;t_0)$ are defined to be the scattering 
coefficients which works to be
\begin{equation}
\begin{split}
&a(\zeta;t_0)={\phi}_1(t_0+2T;\zeta),\\
&b(\zeta;t_0)=\kappa {\phi}_2(t_0+2T;\zeta).
\end{split}
\end{equation}
The matrix defined by
\begin{equation}
\Phi(\zeta;t_0)
=\left(\vs{\phi},\kappa\ovl{\vs{\phi}}\right)_{t=t_0+2T}
=\begin{pmatrix}
\phi_1(t_0+2T;\zeta) & \kappa \phi_2^*(t_0+2T;\zeta^*)\\
\phi_2(t_0+2T;\zeta) & \phi_1^*(t_0+2T;\zeta^*)
\end{pmatrix},
\end{equation}
connects the solution at $t=t_0$ to $t=t_0+2T$, therefore, is referred to as the
\emph{transfer} matrix (or the monodromy matrix). It is easy to see that 
$\Wrons(\vs{\phi},\ovl{\vs{\phi}})=\kappa\neq0$ (a consequence
of the fact that $\vs{\phi}$ and $\ovl{\vs{\phi}}$ are linearly independent) 
which yields $\det\Phi=1$.

The next important distinction in the periodic case is the existence of the so
called \emph{Bloch eigenfunctions}, denoted by $\vs{\psi}(t;\zeta)$, which is 
defined by
\begin{equation}
\vs{\psi}(t+2T;\zeta)=\lambda\vs{\psi}(t;\zeta).
\end{equation}
Consider the representation
\begin{equation}
\vs{\psi}(t)=
c(\zeta;t_0)\vs{\phi}(t;\zeta)+
d(\zeta;t_0)\ovl{\vs{\phi}}(t;\zeta),
\end{equation}
which translates into
\begin{equation}
\Phi(\zeta;t_0)\begin{pmatrix}
c\\
d
\end{pmatrix}=
\lambda\begin{pmatrix}
c\\
d
\end{pmatrix}.
\end{equation}
Therefore, a necessary condition for the existence of Bloch eigenfunctions is
that 
\begin{equation}
\det\left[\Phi(\zeta;t_0)-\lambda\sigma_0\right]=\lambda^2-\trace[\Phi]\lambda+1=0.
\end{equation}
The main spectrum of a periodic signal is defined as the set of eigenvalues
$\zeta_i$ such that $\trace[\Phi](\zeta_i)=\pm2$ so that $\lambda=\pm1$. These values of $\lambda$
correspond to Bloch eigenfunctions that are either periodic or anti-periodic.
In the following, we adopt the terminology introduced in Ma and
Ablowitz~\cite{MA1981} for the
characterization of the main spectra. To this end, define 
$a_R$ and $a_I(\zeta)$ to be the real and imaginary parts of $a(\zeta;t_0)$ for
$\zeta\in\field{R}$, respectively, so that their analytic continuation into the complex plane
reads as
\begin{align}
&a_R(\zeta)=\frac{1}{2}\left[a(\zeta;t_0)+\ovl{a}(\zeta;t_0)\right]
=\frac{1}{2}\trace[\Phi](\zeta),\\
&a_I(\zeta;t_0)=\frac{1}{2i}\left[a(\zeta;t_0)-\ovl{a}(\zeta;t_0)\right].
\end{align}
Note that $a_R(\zeta)$ is independent of $t_0$ and satisfies the symmetry
condition $\ovl{a}_R(\zeta)=a_R(\zeta)$ for all $\zeta\in\field{C}$. 
The main spectrum can be characterized based on the zeros of $1-a_R^2(\zeta)$.

\section{The Numerical Scheme}\label{eq:num-scheme}
In this section, we will develop the integrating factor based exponential
Runge-Kutta method for the numerical solution of the 
Zakharov-Shabat scattering problem. We also include a general discussion of the
exponential linear multistep method which was presented in~\cite{V2018LPT} and
employed as a benchmarking tool in~\cite{V2018BL}.

Following~\cite{V2017INFT1,V2018LPT}, in order to develop the numerical 
scheme, we begin with the transformation 
$\tilde{\vv{v}}=e^{i\sigma_3\zeta t}\vv{v}$
so that~\eqref{eq:ZS-prob} becomes
\begin{equation}\label{eq:exp-int}
\tilde{\vv{v}}_t=\wtilde{U}\tilde{\vv{v}},\quad
\wtilde{U}=e^{i\sigma_3\zeta t}Ue^{-i\sigma_3\zeta t},
\end{equation}
with $\wtilde{U}_{11}=\wtilde{U}_{22}=0,\,\wtilde{U}_{12}=q(t)e^{2i\zeta t}$ and
$\wtilde{U}_{21}=r(t)e^{-2i\zeta t}$. 
Let us remark that the general theory of the Runge-Kutta method and linear multistep
methods can be found in classic texts such as Hairer~\et~\cite{HNW1993} and
Butcher~\cite{Butcher2003}.
\subsection{Runge-Kutta Method}
Let the Butcher tableau of a $s$-stage Runge-Kutta method be given by
\begin{equation}
\renewcommand\arraystretch{1.2}
\begin{array}
{c|cccc}
c_1     & a_{11} & a_{12} & \ldots & a_{1s}\\
c_2     & a_{21} & a_{22} & \ldots & a_{2s}\\
\vdots  & \vdots & \vdots & \ddots & \vdots\\
c_s     & a_{s1} & a_{s2} & \ldots & a_{ss}\\
\hline
        & b_1    & b_2 & \ldots    & b_s
\end{array}\equiv 
\begin{array}
{c|c}
\vv{c}  & A \\
\hline
        & \vv{b}^{\tp}
\end{array}.
\end{equation}
Let the step size be $h>0$ and the quantities $c_j\in[0,1]$ be ordered 
so that the nodes within the step can be stated as 
$t_n\leq t_n+c_1h\leq t_n+c_2h\leq\ldots\leq t_n+c_sh\leq t_{n+1}$. For the potential
sampled at these nodes, we use the convention
$Q_{n+c_k}=hq(t_n+c_kh)$, $R_{n+c_k}=hr(t_n+c_kh)$ and
$\wtilde{U}_{n+c_k}=\wtilde{U}(t_n+c_kh)$.
In order for the resulting discrete system to be amenable to FFT-based fast polynomial
arithmetic, it is sufficient to have each of the $c_i$'s to be taken from a
set of uniformly distributed nodes in $[0,1]$. Introducing the intermediate stage quantities
$\tilde{\vv{v}}_{n,k}$ for $k=1,2,\ldots,s$, we have
\begin{equation}
\tilde{\vv{v}}_{n,j}=\tilde{\vv{v}}_{n}+h\sum_{k=1}^sa_{jk}\wtilde{U}_{n+c_k}\tilde{\vv{v}}_{n,k},\quad
j=1,2,\ldots,s,
\end{equation}
with the final update given by
\begin{equation}
\tilde{\vv{v}}_{n+1}=\tilde{\vv{v}}_{n}+h\sum_{k=1}^sb_{k}\wtilde{U}_{n+c_k}\tilde{\vv{v}}_{n,k}.
\end{equation}
Defining the augmented intermediate stage vector 
\begin{equation}
\wtilde{\vs{\Upsilon}}_n
=\left(\tilde{\vv{v}}_{n,1},\tilde{\vv{v}}_{n,2},\ldots,\tilde{\vv{v}}_{n,s}\right)^{\tp}\in\field{C}^{2s},
\end{equation}
and introducing the block-diagonal matrix
\begin{equation}
\wtilde{D}_n=\diag\left(h\wtilde{U}_{n+c_1},\ldots,h\wtilde{U}_{n+c_s}\right)\in\field{C}^{2s\times2s},
\end{equation}
we have
\begin{equation}\label{eq:linear-sys-RK}
\begin{pmatrix}
I_{2s}-(A\otimes\sigma_0)\wtilde{D}_n & 0_{2s\times2}\\
-(\vv{b}^{\tp}\otimes\sigma_0)\wtilde{D}_n & \sigma_0
\end{pmatrix}
\begin{pmatrix}
\tilde{\vs{\Upsilon}}_n\\
\tilde{\vv{v}}_{n+1}
\end{pmatrix}=\mathbf{1}_{s+1}\otimes\tilde{\vv{v}}_{n},
\end{equation}
where `$\otimes$' stands for the Kronecker-product of two matrices and
$\vv{1}_{s+1}$ denotes a column vector of size $s+1$ with $1$ as its entries.
Our goal in the following is to solve the linear system using the Cramer's rule
and find the transfer matrix relationship between $\tilde{\vv{v}}_n$ and
$\tilde{\vv{v}}_{n+1}$. To this end, consider
\begin{equation}
\Delta_{n+1}=\det\begin{pmatrix}
I_{2s}-(A\otimes\sigma_0)\wtilde{D}_n & 0_{2s\times2}\\
-(\vv{b}^{\tp}\otimes\sigma_0)\wtilde{D}_n & \sigma_0 
\end{pmatrix}
=\det\left(I_{2s}-(A\otimes\sigma_0)\wtilde{D}_n\right).
\end{equation}
In order to obtain $v_1^{(n+1)}$ according to the Cramer's rule, we must
consider the determinant
\begin{equation*}
\det\begin{pmatrix}
I_{2s}-(A\otimes\sigma_0)\wtilde{D}_n & \mathbf{1}_{s}\otimes\tilde{\vv{v}}_{n}
& \vv{0}_{2s} \\
-(\vv{b}^{\tp}\otimes\sigma_0)\wtilde{D}_n & \tilde{\vv{v}}_n & (0,1)^{\tp} 
\end{pmatrix}
=\det\begin{pmatrix}
\wtilde{\Gamma} & \mathbf{0}_{2s} & -\mathbf{1}_{s}\otimes(0,1)^{\tp}  \\
-(\vv{b}^{\tp}\otimes\sigma_0)\wtilde{D}_n & \tilde{\vv{v}}_n &  (0,1)^{\tp} 
\end{pmatrix},
\end{equation*}
where $\vv{0}_{2s}$ denotes a column vector of size $2s$ with $0$ as its
entries and
\begin{equation}
\wtilde{\Gamma}=I_{2s}-\left\{(A-\mathbf{1}_{s}\vv{b}^{\tp})\otimes\sigma_0\right\}\wtilde{D}_n.
\end{equation} 
Now, expanding along the second last column, we have
\begin{equation}\label{eq:tm1}
\begin{split}
\tilde{v}_1^{(n+1)}
&=+\frac{\tilde{v}_1^{(n)}}{\Delta_{n+1}}
\det\begin{pmatrix}
\wtilde{\Gamma} & -\mathbf{1}_{s}\otimes(0,1)^{\tp}\\
-(\vv{b}^{\tp}\otimes(0,1))\wtilde{D}_n &  1 
\end{pmatrix}\\
&\quad-\frac{\tilde{v}_2^{(n)}}{\Delta_{n+1}}
\det\begin{pmatrix}
\wtilde{\Gamma} & -\mathbf{1}_{s}\otimes(0,1)^{\tp}\\
-(\vv{b}^{\tp}\otimes(1,0))\wtilde{D}_n &  0 
\end{pmatrix}.
\end{split}
\end{equation}
Next, in order to obtain $v_2^{(n+1)}$ according to the Cramer's rule, we must
consider the determinant
\begin{equation*}
\det\begin{pmatrix}
I_{2s}-(A\otimes\sigma_0)\wtilde{D}_n 
& \vv{0}_{2s}
& \mathbf{1}_{s}\otimes\tilde{\vv{v}}_{n}\\
-(\vv{b}^{\tp}\otimes\sigma_0)\wtilde{D}_n 
& (1,0)^{\tp}
& \tilde{\vv{v}}_n 
\end{pmatrix}
=\det\begin{pmatrix}
\wtilde{\Gamma}& -\mathbf{1}_{s}\otimes(1,0)^{\tp}
& \mathbf{0}_{2s}   \\
-(\vv{b}^{\tp}\otimes\sigma_0)\wtilde{D}_n 
&  (1,0)^{\tp}
& \tilde{\vv{v}}_n  
\end{pmatrix}.
\end{equation*}
Expanding along the second last column, we have
\begin{equation}\label{eq:tm2}
\begin{split}
\tilde{v}_2^{(n+1)}
&=-\frac{\tilde{v}_1^{(n)}}{\Delta_{n+1}}\det\begin{pmatrix}
\wtilde{\Gamma} & -\mathbf{1}_{s}\otimes(1,0)^{\tp}  \\
-(\vv{b}^{\tp}\otimes(0,1))\wtilde{D}_n &  0 
\end{pmatrix}\\
&\quad+\frac{\tilde{v}_2^{(n)}}{\Delta_{n+1}}\det\begin{pmatrix}
\wtilde{\Gamma} & -\mathbf{1}_{s}\otimes(1,0)^{\tp}  \\
-(\vv{b}^{\tp}\otimes(1,0))\wtilde{D}_n &  1 
\end{pmatrix}.
\end{split}
\end{equation}
Let us introduce the transfer matrix $\wtilde{M}_{n+1}$ so that
\begin{equation}\label{eq:ZS-TM0}
\tilde{\vv{v}}_{n+1}=\Delta_{n+1}^{-1}\wtilde{M}_{n+1}\tilde{\vv{v}}_n.
\end{equation} 
Now, the elements of the transfer matrix
can be obtained from the equations~\eqref{eq:tm1} and~\eqref{eq:tm2} which after
some simplification read as follows: The diagonal elements work out to be
\begin{equation}
\begin{split}
\wtilde{M}^{(n+1)}_{11}
&=\det\left[\wtilde{\Gamma}
-\left\{(\mathbf{1}_{s}\vv{b}^{\tp})\otimes
\left(\frac{\sigma_0-\sigma_3}{2}\right)\right\}\wtilde{D}_n\right],\\
\wtilde{M}^{(n+1)}_{22}
&=\det\left[\wtilde{\Gamma}
-\left\{(\mathbf{1}_{s}\vv{b}^{\tp})\otimes
\left(\frac{\sigma_0+\sigma_3}{2}\right)\right\}\wtilde{D}_n\right].
\end{split}
\end{equation}
The off diagonal elements are given by
\begin{equation}
\begin{split}
\wtilde{M}^{(n+1)}_{12}&=\det\wtilde{\Gamma}
-\det\left[\wtilde{\Gamma}
-\left\{(\mathbf{1}_{s}\vv{b}^{\tp})\otimes
\left(\frac{\sigma_1-i\sigma_2}{2}\right)\right\}\wtilde{D}_n\right],\\
\wtilde{M}^{(n+1)}_{21}&=\det\wtilde{\Gamma}
-\det\left[\wtilde{\Gamma}
-\left\{(\mathbf{1}_{s}\vv{b}^{\tp})
\otimes\left(\frac{\sigma_1+i\sigma_2}{2}\right)\right\}\wtilde{D}_n\right].
\end{split}
\end{equation}
Consider the relations
\begin{equation}\label{eq:ppty-pot}
\begin{split}
&\wtilde{U}_{n+c_i}=\sigma_1\wtilde{U}^{\tp}_{n+c_i}\sigma_1,\\
&\wtilde{D}_{n}    =(I_{s}\otimes\sigma_1)\wtilde{D}^{\tp}_{n}(I_{s}\otimes\sigma_1).
\end{split}
\end{equation}
An interesting consequence of this is
\begin{equation}
\begin{split}
\Delta_{n+1}
&=\det\left(I_{2s}-(I_{s}\otimes\sigma_1)(A\otimes\sigma_0)(I_{s}\otimes\sigma_1)\wtilde{D}^{\tp}_n\right)\\
&=\det\left(I_{2s}-(I_{s}\otimes\sigma_1)(A\otimes\sigma_1)\wtilde{D}^{\tp}_n\right)\\
&=\det\left(I_{2s}-(A\otimes\sigma_0)\wtilde{D}^{\tp}_n\right).
\end{split}
\end{equation}
This also turns out to be true of $\det\wtilde{\Gamma}$:
\begin{equation}
\det\wtilde{\Gamma}=\det\left(I_{2s}-\left\{(A-\mathbf{1}_{s}\vv{b}^{\tp})\otimes\sigma_0\right\}
\wtilde{D}^{\tp}_n\right).
\end{equation} 
Consider
\begin{equation}
\begin{split}
(I_{s}\otimes\sigma_1)\left\{(\mathbf{1}_{s}\vv{b}^{\tp})\otimes
\left(\frac{\sigma_0-\sigma_3}{2}\right)\right\}(I_{s}\otimes\sigma_1)
&=(\mathbf{1}_{s}\vv{b}^{\tp})\otimes\left\{
\sigma_1\left(\frac{\sigma_0-\sigma_3}{2}\right)\sigma_1\right\}\\
&=(\mathbf{1}_{s}\vv{b}^{\tp})\otimes\left(\frac{\sigma_0+\sigma_3}{2}\right),
\end{split}
\end{equation}
and,
\begin{equation}
\begin{split}
(I_{s}\otimes\sigma_1)\left\{(\mathbf{1}_{s}\vv{b}^{\tp})\otimes
\left(\frac{\sigma_1-i\sigma_2}{2}\right)\right\}(I_{s}\otimes\sigma_1)
&=(\mathbf{1}_{s}\vv{b}^{\tp})\otimes\left\{
\sigma_1\left(\frac{\sigma_1-i\sigma_2}{2}\right)\sigma_1\right\}\\
&=(\mathbf{1}_{s}\vv{b}^{\tp})\otimes\left(\frac{\sigma_1+i\sigma_2}{2}\right).
\end{split}
\end{equation}
These relations allow us to conclude:
\begin{lemma}
The transfer matrix $\wtilde{M}_{n+1}$ seen as a function of $\wtilde{D}_n$ from
a subspace of $\field{C}^{2s\times2s}$, which satisfies the
property~\eqref{eq:ppty-pot}, to $\field{C}^{2\times2}$ satisfies the following relations:
\begin{align*}
\wtilde{M}^{(n+1)}_{11}(\wtilde{D}_n)=&\wtilde{M}^{(n+1)}_{22}(\wtilde{D}^{\tp}_n),\\
\wtilde{M}^{(n+1)}_{12}(\wtilde{D}_n)=&\wtilde{M}^{(n+1)}_{21}(\wtilde{D}^{\tp}_n).
\end{align*}
\end{lemma}
This lemma may prove useful in verifying the output of any computer algebra
system that is used for computing the transfer matrix. 

Diagonally implicit RK methods are characterized by a lower triangular matrix $A$. In this 
case $\Delta_{n+1}$ can be computed explicitly:
\begin{equation}
\Delta_{n+1}=\prod_{k=1}^{s}(1+a_{kk}Q_{n+c_k}R_{n+c_k}).
\end{equation}
For any matrix $A$, let $\|A\|$ stand for the spectral norm unless otherwise
stated.
\begin{lemma}\label{lemma:Delta-zero-free}
Given samples of the potential, $q_{n+c_1}, q_{n+c_2},\ldots,q_{n+c_s}$, 
assumed to be bounded, there exits a constant $h_0>0$ such that 
\[
\Delta_{n+1}(\zeta)\neq0,\quad\forall\zeta\in\field{R},
\]
if $0\leq h\leq h_0$. 
\end{lemma}
\begin{proof}
In order to prove the lemma, it suffices to show that for some $h_0>0$, the
spectral norm of $(A\otimes\sigma_0)\wtilde{D}_n$ is less than unity for all
$h\leq h_0$ and $\zeta\in\field{R}$. To this end, consider
\[
\|(A\otimes\sigma_0)\wtilde{D}_n\|\leq\|A\|\|\wtilde{D}_n\|.
\]
Noting that
\[
\wtilde{D}_n\wtilde{D}_n^{\dagger}=\diag(|Q_{n+c_1}|^2,\ldots,|Q_{n+c_s}|^2)\otimes\sigma_0,
\]
it follows that $\|\wtilde{D}_n\|\leq\max_{1\leq k\leq s}|Q_{n+c_k}|$. If
$|q|<\omega$, we may write $\|\wtilde{D}_n\|<\omega h$ so that
\[
\|(A\otimes\sigma_0)\wtilde{D}_n\|< h (\omega\|A\|).
\]
The proof follows if we choose $h_0<(\omega\|A\|)^{-1}$.
\end{proof}
In fact, it is also possible to show that by choosing $h$
sufficiently small for given $\zeta\in\field{C}$, one can ensure that
\[
\Delta_{n+1}(\zeta)\neq0.
\]
We defer this discussion to later part of this section.
In order to study the stability of the method in the sense of~\cite{Gautschi2012} 
for fixed $\zeta\in\field{R}$, we write the 
update relation over one-step as
\begin{equation}\label{eq:one-step-RK}
\wtilde{\vv{v}}_{n+1}=\wtilde{\vv{v}}_{n}+h\Lambda_{n+1}\tilde{\vv{v}}_{n};
\end{equation}
then, the boundedness of the matrix
\begin{equation}
\Lambda_{n+1}(t_n;h)
=\frac{1}{h}(\vv{b}^{\tp}\otimes\sigma_0)
\wtilde{D}_n[I_{2s}-(A\otimes\sigma_0)\wtilde{D}_n]^{-1}(\mathbf{1}_s\otimes\sigma_0),
\end{equation}
as $h\rightarrow0$ is sufficient to guarantee the stability of this one-step 
method~\cite{Gautschi2012}. Observing that
\[
\|\Lambda_{n+1}\|\leq
\frac{\sqrt{s}\|\vv{b}\|_2\|\wtilde{D}_n\|/h}{1-\|A\|\|\wtilde{D}_n\|}
\leq\frac{\sqrt{s}\|\vv{b}\|_2\|\omega}{1-\omega h_0\|A\|},
\]
for all $h\leq h_0$ where $\omega$ and $h_0$ are as defined in the proof 
of Lemma~\ref{lemma:Delta-zero-free} so that
$\|\wtilde{D}_n\|<\omega h$. The method~\eqref{eq:one-step-RK} is also
\emph{consistent}, i.e., the truncation error goes to $0$ as $h\rightarrow0$,
follows from
\begin{equation*}
\Lambda_{n+1}(t_n;0)
=(\vv{b}^{\tp}\otimes\sigma_0)(\mathbf{1}_s\otimes\wtilde{U}_n)
=(\vv{b}^{\tp}\mathbf{1}_s)\otimes\wtilde{U}_n,
\end{equation*}
provided $(\vv{b}^{\tp}\mathbf{1}_s)=\sum_{j=1}^sb_j=1$. Consistency and
stability imply convergence; therefore, the method~\eqref{eq:one-step-RK} is
convergent under this condition.
\begin{prop}\label{prop:convg-zeta-real}
Let $q$ be bounded over its support which is assumed to be compact. Then, 
for fixed $\xi\in\field{R}$, the method defined by~\eqref{eq:one-step-RK} 
is convergent if $\sum_{j=1}^sb_j=1$.
\end{prop}

The transfer matrix relation obtained thus far does not explicitly make it
evident that its entries are rationals. To this end, we carry out further
simplification so that the rational structure of the transfer matrix becomes
self-evident. Let
\begin{equation}
h\wtilde{U}_{n+c_k}=e^{i\sigma_3\zeta
t_n}h\mathcal{U}_{n+c_k}e^{-i\sigma_3\zeta t_n},
\end{equation}
and
\begin{equation}
D_n=\diag\left(h\mathcal{U}_{n+c_1},\ldots,h\mathcal{U}_{n+c_s}\right)\in\field{C}^{2s\times2s},
\end{equation}
so that
\begin{equation}
\wtilde{D}_n=(I_s\otimes e^{i\zeta t_n\sigma_3})D_n
(I_s\otimes e^{-i\zeta t_n\sigma_3}).
\end{equation}
The linear system in~\eqref{eq:linear-sys-RK} can now be written as
\begin{equation}\label{eq:linear-sys-RK-mod}
\begin{pmatrix}
I_{2s}-(A\otimes\sigma_0)D_n & 0_{2s\times2}\\
-(\vv{b}^{\tp}\otimes\sigma_0)D_n & \sigma_0
\end{pmatrix}
\begin{pmatrix}
\vs{\Upsilon}_n\\
e^{i\zeta h\sigma_3}\vv{v}_{n+1}
\end{pmatrix}\\
=\begin{pmatrix}
\mathbf{1}_{s}\otimes\vv{v}_{n}\\
\vv{v}_{n}
\end{pmatrix},
\end{equation}
where $\vs{\Upsilon}_n
=(I_s\otimes e^{-i\zeta t_n\sigma_3})\wtilde{\vs{\Upsilon}}_n$.
Define
\begin{equation}
{\Gamma}=I_{2s}-\left\{(A-\mathbf{1}_{s}\vv{b}^{\tp})\otimes\sigma_0\right\}{D}_n.
\end{equation} 
Then the transfer matrix relation in~\eqref{eq:ZS-TM0} can be stated as
\begin{equation}\label{eq:ZS-TM}
\vv{v}_{n+1}=\frac{e^{-i\zeta h}}{\Delta_{n+1}}M_{n+1}\vv{v}_n,
\end{equation}
where we note that
\begin{equation}
\Delta_{n+1}
=\det\left(I_{2s}-(A\otimes\sigma_0)\wtilde{D}_n\right)
=\det\left(I_{2s}-(A\otimes\sigma_0)D_n\right),
\end{equation}
and the diagonal elements of the transfer matrix are
\begin{equation}
\begin{split}
{M}^{(n+1)}_{11}
&=\det\left[{\Gamma}
-\left\{(\mathbf{1}_{s}\vv{b}^{\tp})\otimes
\left(\frac{\sigma_0-\sigma_3}{2}\right)\right\}{D}_n\right],\\
{M}^{(n+1)}_{22}
&=\det\left[{\Gamma}
-\left\{(\mathbf{1}_{s}\vv{b}^{\tp})\otimes
\left(\frac{\sigma_0+\sigma_3}{2}\right)\right\}{D}_n\right]e^{2i\zeta h},
\end{split}
\end{equation}
while the off-diagonal elements are given by
\begin{equation}
\begin{split}
{M}^{(n+1)}_{12}&=\det{\Gamma}
-\det\left[{\Gamma}
-\left\{(\mathbf{1}_{s}\vv{b}^{\tp})\otimes
\left(\frac{\sigma_1-i\sigma_2}{2}\right)\right\}{D}_n\right],\\
{M}^{(n+1)}_{21}&=
e^{2i\zeta h}\det{\Gamma}-e^{2i\zeta h}\det\left[{\Gamma}
-\left\{(\mathbf{1}_{s}\vv{b}^{\tp})
\otimes\left(\frac{\sigma_1+i\sigma_2}{2}\right)\right\}{D}_n\right].
\end{split}
\end{equation}
\begin{lemma}
Let $\kappa=-1$. Then, the transfer matrix relation between $\vv{v}_{n}(\zeta)$ and
$\vv{v}_{n+1}(\zeta)$ is identical to that between $\sigma_2\vv{v}^*_{n}(\zeta^*)$ and
$\sigma_2\vv{v}^*_{n+1}(\zeta^*)$. In other words, the transfer matrix
$M_{n+1}(\zeta)$ in~\eqref{eq:ZS-TM} satisfies the relation
\[
\sigma_2 M_{n+1}^*(\zeta^*)\sigma_2=e^{-2i\zeta h}M_{n+1}(\zeta),
\]
with
$\Delta^*_{n+1}(\zeta^*)=\Delta_{n+1}(\zeta)$. Further, for $\xi\in\field{R}$,
\[
\|M_{n+1}(\xi)\|=\frac{\sqrt{|\det M_{n+1}(\xi)|}}{|\Delta(\xi)|}.
\]
\end{lemma}
\begin{proof}
For $\kappa=-1$, the symmetry property
\begin{equation*}
\begin{split}
h\mathcal{U}^*_{n+c_k}(\zeta^*)&=
\begin{pmatrix}
0 & -R_{n+c_k}e^{-2i\zeta c_kh}\\
-Q_{n+c_k}e^{2i\zeta c_kh}& 0
\end{pmatrix}\\
&=\sigma_2h\mathcal{U}_{n+c_k}(\zeta)\sigma_2,
\end{split}
\end{equation*}
allows us to conclude
\begin{equation}
{D}^*_{n}(\zeta^*)=(I_s\otimes\sigma_2){D}_{n}(\zeta)(I_s\otimes\sigma_2).
\end{equation}
Changing
$\zeta$ to $\zeta^*$ in~\eqref{eq:linear-sys-RK-mod} and taking the complex
conjugate of both sides of the equation, we have
\begin{align*}
&\left[I_{2s}-(A\otimes\sigma_0)(I_s\otimes\sigma_2){D}_{n}(\zeta)(I_s\otimes\sigma_2)\right]\vs{\Upsilon}^*_n
=\mathbf{1}_{s}\otimes\vv{v}^*_{n},\\
&\left[I_{2s}-(A\otimes\sigma_0){D}_{n}(\zeta)\right]\breve{\vs{\Upsilon}}_n
=\mathbf{1}_{s}\otimes(\sigma_ 2\vv{v}_{n}^*),
\end{align*}
where $\breve{\vs{\Upsilon}}_n=(I_s\otimes\sigma_2)\vs{\Upsilon}^*_n$. Further,
\[
-(\vv{b}^{\tp}\otimes\sigma_0)D_n(\zeta)\breve{\vs{\Upsilon}}_n
+e^{-i\zeta h\sigma_3}\sigma_2\vv{v}^*_{n+1}=\sigma_2\vv{v}^*_{n},
\]
shows that the linear system connecting $\sigma_2\vv{v}^*_{n+1}$ and
$\sigma_2\vv{v}^*_{n}$ with $\zeta^*$ as the spectral parameter is identical to
that between $\vv{v}_{n+1}$ and
$\vv{v}_{n}$ with $\zeta$ as the spectral parameter.

The foregoing conclusion can also be drawn from the structure of the transfer
matrix $M_{n+1}(\zeta)$. Consider
\begin{align*}
\Delta^*_{n+1}(\zeta^*)
&=\det\left(I_{2s}-(A\otimes\sigma_0)(I_s\otimes\sigma_2){D}_{n}(\zeta)(I_s\otimes\sigma_2)\right)\\
&=\det\left(I_{2s}-(I_s\otimes\sigma_2)(A\otimes\sigma_0)(I_s\otimes\sigma_2){D}_{n}(\zeta)\right)\\
&=\Delta_{n+1}(\zeta).
\end{align*}
Next,
\begin{align*}
{\Gamma}^*(\zeta^*)
&=I_{2s}-\left\{(A-\mathbf{1}_{s}\vv{b}^{\tp})\otimes\sigma_0\right\}
(I_s\otimes\sigma_2){D}_{n}(\zeta)(I_s\otimes\sigma_2)\\
&=(I_s\otimes\sigma_2)\Gamma(\zeta)(I_s\otimes\sigma_2),
\end{align*} 
which implies 
\[
M^{(n+1)*}_{11}(\zeta^*)=M^{(n+1)}_{22}(\zeta)e^{-2i\zeta h}.
\]
Now writing ${M}^{(n+1)}_{12}$ in the form 
\begin{equation*}
{M}^{(n+1)}_{12}(\zeta)=
(\det{\Gamma})(\vv{b}^{\tp}\otimes(1,0)){D}_n(\zeta)(\Gamma^{-1})(\vs{1}_s\otimes(0,1))
\end{equation*}
we have
\begin{equation*}
\begin{split}
{M}^{(n+1)*}_{12}(\zeta^*)
&=-(\det{\Gamma})(\vv{b}^{\tp}\otimes(0,1)){D}_n(\zeta)(\Gamma^{-1})(\vs{1}_s\otimes(1,0))\\
&=-{M}^{(n+1)}_{21}(\zeta)e^{-2i\zeta h}.
\end{split}
\end{equation*}

\end{proof}

\begin{lemma}
Let $\kappa=+1$. Then, the transfer matrix relation between $\vv{v}_{n}(\zeta)$ and
$\vv{v}_{n+1}(\zeta)$ is identical to that between $\sigma_1\vv{v}^*_{n}(\zeta^*)$ and
$\sigma_1\vv{v}^*_{n+1}(\zeta^*)$. In other words, the transfer matrix
$M_{n+1}(\zeta)$ in~\eqref{eq:ZS-TM} satisfies the relation
\[
\sigma_1 M_{n+1}^*(\zeta^*)\sigma_1=e^{-2i\zeta h}M_{n+1}(\zeta),
\]
with
$\Delta^*_{n+1}(\zeta^*)=\Delta_{n+1}(\zeta)$. Further, for $\xi\in\field{R}$,
\[
\|M_{n+1}(\xi)\|=\frac{|{M}^{(n+1)}_{11}(\xi)|+|{M}^{(n+1)}_{12}(\xi)|}{|\Delta(\xi)|}.
\]
\end{lemma}
\begin{proof}
For $\kappa=+1$, the symmetry property
\begin{equation}
\begin{split}
h\mathcal{U}^*_{n+c_k}(\zeta^*)&=
\begin{pmatrix}
0 & R_{n+c_k}e^{-2i\zeta c_kh}\\
Q_{n+c_k}e^{2i\zeta c_kh}& 0
\end{pmatrix}
=\sigma_1h\mathcal{U}_{n+c_k}(\zeta)\sigma_1,
\end{split}
\end{equation}
allows us to conclude
\begin{equation}
{D}^*_{n}(\zeta^*)=(I_s\otimes\sigma_1){D}_{n}(\zeta)(I_s\otimes\sigma_1).
\end{equation}
The rest of the proof is similar to that of the preceding lemma, therefore, it
is being omitted for the sake of brevity of presentation.
\end{proof}
Now, we would like to address the stability of the method~\eqref{eq:one-step-RK} for fixed
$\zeta=\xi+i\eta\in\field{C}$. Rewrite the update relation as
\begin{equation}\label{eq:one-step-RK-complex}
{\vv{v}}_{n+1}=e^{-i\zeta h\sigma_3}
\left[{\vv{v}}_{n}+h\Lambda_{n+1}\vv{v}_{n}\right],
\end{equation}
where
\begin{equation}
\Lambda_{n+1}(t_n;h)
=\frac{1}{h}(\vv{b}^{\tp}\otimes\sigma_0)
{D}_n[I_{2s}-(A\otimes\sigma_0){D}_n]^{-1}(\mathbf{1}_s\otimes\sigma_0).
\end{equation}
Now, taking into account that $\|D_n\|\leq\omega he^{2|\eta|h}$ 
where $\omega$ is as defined in the proof 
of Lemma~\ref{lemma:Delta-zero-free} and choosing
$h_0<(2|\eta|+\omega\|A\|)^{-1}$, we have
\[
\|\Lambda_{n+1}\|\leq
\frac{\sqrt{s}\|\vv{b}\|_2\|{D}_n\|/h}{1-\|A\|\|{D}_n\|}
\leq\frac{\sqrt{s}\|\vv{b}\|_2\|\omega}{1-\omega h_0\|A\|},
\]
for all $h\leq h_0$. Under these conditions, we also have $\Delta_{n+1}(\zeta)\neq0$.
The proposition~\ref{prop:convg-zeta-real} can now be modified as
\begin{prop}\label{prop:convg-complex-real}
Let $q$ be bounded over its support which is assumed to be compact. Then, for 
fixed $\zeta\in\field{C}$ such that $|\Im\zeta|<\infty$, the method defined
by~\eqref{eq:one-step-RK-complex} is convergent if $\sum_{j=1}^sb_j=1$.
\end{prop}
These observations lead to the conclusion that in order to evaluate the Jost solutions at
any point $\zeta$ in the complex plane, the step size must be chosen to be much
smaller than that in the case of any point $\xi$ on the real axis. The determination of 
the discrete spectrum is, therefore, numerically more challenging than that of
the continuous spectrum.

Next we look at the recurrence relation for the Wronskian for real values of the
spectral parameter. To this end, we set $\zeta=\xi\in\field{R}$. The transfer matrix
relation~\eqref{eq:ZS-TM} can also be written for the matrix solution of the ZS
problem. Let $v_{n}=(\vs{\phi}_n,\vs{\psi}_n)$ so that
\begin{equation}\label{eq:mat-ZS-TM}
{v}_{n+1}=\frac{e^{-i\xi h}}{\Delta_{n+1}}M_{n+1}{v}_n.
\end{equation}
The recurrence relation for the Wronskian then works out to be
\begin{equation}
W_{n+1}=\frac{e^{-2i\xi h}}{\Delta^2_{n+1}(\xi)}\det(M_{n+1})W_n
=\frac{|M^{(n+1)}_{11}(\xi)|^2-\kappa|M^{(n+1)}_{12}(\xi)|^2}{\Delta^2_{n+1}(\xi)}W_n.
\end{equation}
Given that $\vs{\phi}_n$ and $\vs{\psi}_n$ are linearly independent solutions,
their Wronskian must not be zero. For sufficiently small $h$, one can ensure that
$|M^{(n+1)}_{11}(\xi)|^2-\kappa|M^{(n+1)}_{12}(\xi)|^2\neq0$ and
$\Delta^2_{n+1}(\xi)\neq0$. For $\kappa=-1$, we have 
$|M^{(n+1)}_{11}(\xi)|^2-\kappa|M^{(n+1)}_{12}(\xi)|^2\neq0$ for all values of
$h$.

Finally, let us show that the entries of the transfer matrix are rationals.
First, let us observe that
\begin{equation}
h\mathcal{U}_{n+c_k}=\begin{pmatrix}
0 & Q_{n+c_k}e^{2i\zeta c_kh}\\
R_{n+c_k}e^{-2i\zeta c_kh}& 0
\end{pmatrix}.
\end{equation}
Let $\nu\in\field{Z}_+$ be such that $c_k=n_k/\nu$ where $n_k\in\field{Z}_+\cup\{0\}$. Then
putting $z=\exp(i\zeta h/\nu)$, we have
\begin{equation}
h\mathcal{U}_{n+c_k}(z^2)=
\begin{pmatrix}
0 & Q_{n+c_k}z^{2n_k}\\
R_{n+c_k}z^{-2n_k}& 0
\end{pmatrix},
\end{equation}
so that $D_n=D_n(z^2)$ with finite powers of $z^2$. This shows that $M_{n+1}=M_{n+1}(z^2)$ and
$\Delta_{n+1}=\Delta_{n+1}(z^2)$ with finite powers of $z^2$.

\subsubsection{Computation of discrete scattering coefficients}
Let the computational domain be $\Omega=[T_1,T_2]$ and set 
\begin{equation}
\left\{\begin{aligned}
&2T=T_2-T_1,\\
&2\ovl{T}= T_1+T2. 
\end{aligned}\right.
\end{equation}
Let the potential $q$ be supported in $\Omega$ or it is
assumed to periodic with period $2T$. Let the number
of of segments be $N_{\text{seg.}}$ so that the number of samples of the
potential becomes $N=\nu N_{\text{seg.}}+1$ (here $\nu+1$ is 
the number of distinct nodes within $[0,1]$) and $h=2T/N_{\text{seg.}}$. Consider 
the Jost solution $\phi(t;\zeta)$:
\begin{equation}\label{eq:phi-init-final}
\vs{\phi}(T_1;\zeta)=
\begin{pmatrix}
1\\
0
\end{pmatrix}e^{-i\zeta T_1},\quad
{\vs{\phi}}(T_2;\zeta)=
\begin{pmatrix}
a(\zeta)e^{-i\zeta T_2}\\
b(\zeta)e^{+i\zeta T_2}
\end{pmatrix},
\end{equation}
Let the cumulative transfer matrix be given by
\begin{equation}
M_{\text{cum..}}
=
\begin{pmatrix}
M^{(\text{cum..})}_{11} & M^{(\text{cum.})}_{12}\\
M^{(\text{cum.})}_{21} & M^{(\text{cum.})}_{22}
\end{pmatrix}
=e^{-2i\zeta T}\frac{M_{N_{\text{seg.}}}\times\ldots\times M_{1}}
{\Delta_{N_{\text{seg.}}}\times\ldots\times \Delta_{1}},
\end{equation}
then, the discrete scattering coefficients $a$ and $b$ are given by
\begin{equation}
\begin{split}
&a_N(z^2) = (z^2)^{\nu\ell}M^{(\text{cum.})}_{11}(z^2),\\
&b_N(z^2) = (z^2)^{-\nu\ovl{\ell}}M^{(\text{cum.})}_{21}(z^2),
\end{split}
\end{equation}
defined for $\Re\zeta\in[-\pi/2h,\pi/2h]$ where $h\ell\nu =T$, 
$h\ovl{\ell}\nu =\ovl{T}$ and $z=e^{i\zeta h/\nu}$. Also, let $h\ell_+=T_2$ and $h\ell_-=-T_1$ so that
\begin{equation}
\left\{\begin{aligned}
&2\ell\nu=\ell_++\ell_-,\\
&2\ovl{\ell}\nu=\ell_+-\ell_-.
\end{aligned}\right.
\end{equation}
The notation above is
consistent with the fact that the transfer matrices have entries that are rational 
functions of $z^2$ .

Let $N_{\text{mat.}}$ be the number of matrices and let $m$ be the number of
coefficients of the polynomials involved in the rational entries of each of the
matrices. In order to use FFT-based fast
polynomial arithmetic~\cite{Henrici1993} to compute the cumulative transfer
matrix, $N_{\text{mat.}}$ must be a power of $2$ and the polynomials involved in
the rational entries must have number of coefficients $m$ that is a power of $2$.
Note that by introducing dummy transfer matrices (which can be the identity matrix)
one can easily choose $N_{\text{mat.}}$ to be a power of $2$. The number of
coefficients $m$ can also be chosen freely by introducing dummy coefficients which
are identically zero. Let $\ovl{N}=mN_{\text{mat.}}$. The complexity of computing the
polynomial coefficients of the discrete scattering coefficient can be worked out
to be $\bigO{\ovl{N}\log^2\ovl{N}}$. The evaluation of the
scattering coefficient at $N'\geq \ovl{N}$ (where $N'$ is a power of $2$) equispaced points 
amounts to an FFT operation; therefore, the complexity of computing the
reflection coefficient works out to be $\bigO{N'\log N'}$. 

In order to represent the Jost solutions, we introduce the
polynomial vector
\begin{equation}\label{eq:poly-vec-rk}
\vv{P}_n(z^2)
=\begin{pmatrix}
P^{(n)}_{1}(z^2)\\
P^{(n)}_{2}(z^2)
\end{pmatrix}
=\sum_{j=0}^{n}
\vv{P}^{(n)}_{j}(z^2)^{j}
=\sum_{j=0}^{n}
\begin{pmatrix}
P^{(n)}_{1,j}\\
P^{(n)}_{2,j}
\end{pmatrix}(z^2)^{j},
\end{equation}
and the polynomial 
\begin{equation}
D_n(z^2) = \sum_{j=0}^{n}D^{(n)}_j (z^2)^j.
\end{equation}
For the methods considered in this paper, the form of the discrete scattering
coefficients can be discussed as follows: Let $M_{n}(z^2)$ have polynomial entries 
of degree $(z^2)^{\nu}$ and
\begin{equation}
\vv{P}_{n+1}(z^2)=\frac{1}{\Theta_{n+1}}M_{n+1}(z^2)\vv{P}_{n}(z^2),
\end{equation}
where $\Theta_{n}$ is defined below.
\begin{itemize}
\item Let $\Delta_n(z^2)z^{-\nu}$ be a polynomial of degree
$(z^2)^{\nu}$ whose constant term is $\Theta_{n}$ so that
\begin{equation}
D_{n+1}(z^2)=\frac{z^{-\nu}}{\Theta_{n+1}}\Delta_{n+1}(z^2)D_{n}(z^2);
\end{equation}
then,
\begin{equation}
\begin{split}
&a_N(z^2) = (z^2)^{\nu\ell}\frac{P^{(\nu N_{\text{seg.}})}_1(z^2)}{D_{\nu N_{\text{seg.}}}(z^2)},\\
&b_N(z^2) = (z^2)^{-\nu\ovl{\ell}}\frac{P_2^{(\nu N_{\text{seg.}})}(z^2)}{D_{\nu N_{\text{seg.}}}(z^2)},
\end{split}
\end{equation}
where we recall $N=\nu N_{\text{seg.}} + 1$. By introducing a dummy coefficient, it
is also justified to write
\begin{equation}
\begin{split}
&a_N(z^2) = (z^2)^{\nu\ell}\frac{P^{(N)}_1(z^2)}{D_{N}(z^2)},\\
&b_N(z^2) = (z^2)^{-\nu\ovl{\ell}}\frac{P_2^{(N)}(z^2)}{D_{N}(z^2)},
\end{split}
\end{equation}
with the domain of definition restricted by $\Re\zeta\in[-\pi/2h,\pi/2h]$.
\item Let $\Delta_n(z^2)=\Theta_n$ which is independent of $z$; then
\begin{equation}
\begin{split}
&a_N(z^2) = {P^{(N)}_1(z^2)},\\
&b_N(z^2) = (z^2)^{-\nu\ell_+}{P_2^{(N)}(z^2)}.
\end{split}
\end{equation}

\end{itemize}

The computation of the discrete eigenvalues requires finding all the zeros
of $a_N(z^2)$ in $\field{C}_+$. We do not address this problem in this paper
and refer the reader to the work of Henrici~\cite{Henrici1993} and, more
recent, Van Barel~\et~\cite{VVVF2010}, Kravanja and Van Barel~\cite{KB2000}. We will instead assume that the
eigenvalues are known by design\footnote{Let us remark that the fastest
non-iterative method available for computing the zeros of a polynomial requires
$\bigO{N^2_{\text{deg.}}}$ operations 
(where $N_{\text{deg.}}$ is the degree of the polynomial) and it is based on finding the eigenvalues
of a companion matrix.}. The method of computation of the norming
constant is adapted from~\cite{V2017INFT1} which introduces an
additional complexity of $\bigO{KN_{\text{seg.}}}$ or, equivalently, $\bigO{KN}$ 
where $K$ is the number of eigenvalues.

Therefore, excluding the cost of computing the discrete eigenvalues, the
total complexity of computing the nonlinear Fourier spectrum to be
$\bigO{N(K+\log^2N)}$.

\subsubsection{One-stage methods}
The simplest example of a one-stage method which achieves second order of
convergence is the \emph{implicit midpoint} (IM) method:
\begin{equation}
\renewcommand\arraystretch{1.2}
\begin{array}
{c|r}
1/2 & {1}/{2}\\
\hline
  & {1}
\end{array}
\end{equation}
In the expanded form, we have the system of equations
\begin{align*}
&\tilde{\vv{v}}_{n,1} = \tilde{\vv{v}}_{n}   
+(h/2)\wtilde{U}_{n+1/2}  \tilde{\vv{v}}_{n,1},\\
&\tilde{\vv{v}}_{n+1} = \tilde{\vv{v}}_{n}   
+h\wtilde{U}_{n+1/2}  \tilde{\vv{v}}_{n,1}.
\end{align*}
Solving this linear system, we have
\begin{equation}
\tilde{\vv{v}}_{n+1}\\
=\left[\sigma_0-(h/2)\wtilde{U}_{n+1/2}\right]^{-1}
\left[\sigma_0+(h/2)\wtilde{U}_{n+1/2}\right]\tilde{\vv{v}}_n.
\end{equation}
Putting $\Theta_{n}=1-\frac{1}{4}Q_nR_n$ and introducing $z=\exp(i\zeta h)$, the 
above expression can be written as
\begin{equation}\label{eq:TM-IM}
\vv{v}_{n+1}=\left(\frac{2-\Theta_{n+1/2}}{\Theta_{n+1/2}}\right)
\begin{pmatrix}
z^{-1}&G_{n+1/2}\\
H_{n+1/2}& z
\end{pmatrix}\vv{v}_{n}
\end{equation}
where
\begin{equation}
G_{n+1/2}=\frac{Q_{n+1/2}}{2-\Theta_{n+1/2}}, \quad H_{n+1/2}=\kappa G^*_{n+1/2}.
\end{equation}
This method resembles the split-Magnus method~\cite{V2017INFT1} and 
it also turns out to be a simplectic method. The
\emph{layer-peeling} formula for this system in terms of the potentials $G$ and
$H$ is identical to that of the split-Magnus method. If $H_{n+1/2}$ is known, the potential
$R_{n+1/2}$ can recovered as
\begin{equation}
R_{n+1/2} = \frac{H_{n+1/2}}{1+\sqrt{1-\kappa|H_{n+1/2}|^2}}.
\end{equation}
It is worth noting that the discrete system based on (exponential) trapezoidal rule presented
in~\cite{V2017INFT1} can also be transformed into a form similar 
to~\eqref{eq:TM-IM} (see~\cite{V2018DDT}). This method also belongs to the class
of Ablowitz-Ladik problems and admits of a discrete Darboux transformation~\cite{V2018DDT}.
 
By employing the transformation $\vv{w}_{n} = e^{i\zeta\sigma_3h/2}\vv{v}_n$,
we obtain
\begin{equation}
\begin{split}
\vv{w}_{n+1}&=
z^{-1}\left(\frac{2-\Theta_{n+1/2}}{\Theta_{n+1/2}}\right)
\begin{pmatrix}
1& z^2G_{n+1/2}\\
H_{n+1/2}& z^2
\end{pmatrix}\vv{w}_n\\
&=z^{-1}M_{n+1}(z^2)\vv{w}_n,
\end{split}
\end{equation}
so that the scattering coefficients can be represented as
\begin{equation}
\begin{split}
& a_{N}(z^2)={P}^{(N)}_1(z^2),\\
& b_{N}(z)=z^{-2\ell_{+}+1}{P}^{(N)}_2(z^2),
\end{split}
\end{equation}
where
\begin{equation}
\vv{P}_{n+1}(z^2)=\left(\frac{2-\Theta_{n+1/2}}{\Theta_{n+1/2}}\right)
M_{n+1}(z^2)\vv{P}_{n}(z^2).
\end{equation}

\begin{rem}[A modified split-Magnus method]
Let us note that a modified form of the split-Magnus method can be 
derived which turns out to be identical to the 
implicit midpoint method described above. Following~\cite{V2017INFT1} and
applying Magnus method with one-point Gauss quadrature to the original 
ZS-problem in~\eqref{eq:ZS-prob}, we obtain
\begin{equation*}
\vv{v}_{n+1}=\exp\left(-i\zeta\sigma_3h+hU_{n+\frac{1}{2}}\right)\vv{v}_n.
\end{equation*}
Splitting the exponential operator as
\begin{equation*}
\exp\left(-i\zeta\sigma_3+hU_{n+\frac{1}{2}}\right)
=\exp\left(-\frac{1}{2}i\zeta\sigma_3h\right)
\exp\left(hU_{n+\frac{1}{2}}\right)
\exp\left(-\frac{1}{2}i\zeta\sigma_3h\right)
+ \bigO{h^3},
\end{equation*}
and setting 
$2\Gamma =\sqrt{Q_{n+\frac{1}{2}}R_{n+\frac{1}{2}}}$, we have
\begin{align*}
\exp\left(hU_{n+\frac{1}{2}}\right)
&=\begin{pmatrix}
\cosh2\Gamma & Q_{n+\frac{1}{2}}\frac{\sinh2\Gamma}{2\Gamma}\\
R_{n+\frac{1}{2}}\frac{\sinh2\Gamma}{2\Gamma} & \cosh2\Gamma\\
\end{pmatrix}\\
&=\left(\frac{1+\tanh^2{\Gamma}}{1-\tanh^2{\Gamma}}\right)
\begin{pmatrix}
1 &
Q_{n+\frac{1}{2}}\frac{(\tanh\Gamma)/\Gamma}{1+\tanh^2{\Gamma}}\\
R_{n+\frac{1}{2}}\frac{(\tanh\Gamma)/\Gamma}{1+\tanh^2{\Gamma}} & 1\\
\end{pmatrix}\\
&=\left(\frac{1+\Gamma^2}{1-\Gamma^2}\right)
\begin{pmatrix}
1& \frac{Q_{n+\frac{1}{2}}}{1+\Gamma^2}\\
\frac{R_{n+\frac{1}{2}}}{1+\Gamma^2} & 1\\
\end{pmatrix}+\bigO{h^3}.
\end{align*}
From here it is evident that the resulting transfer matrix 
becomes identical to that of the implicit midpoint method.
\end{rem}

\subsubsection{Two-stage methods}
We consider two examples of a two-stage, second order, method which are diagonally
implicit. The Butcher tableau for the two-stage
Lobatto IIIA and Lobatto IIIB methods are
given by
\begin{equation}
\renewcommand\arraystretch{1.2}
\begin{array}
{c|rr}
0 &       0 &        0 \\
1 & {1}/{2} &  {1}/{2} \\
\hline
  & {1}/{2} &  {1}/{2} 
\end{array}
\qquad
\begin{array}
{c|rr}
0 &  1/2    &  0  \\
1 & {1}/{2} &  0 \\
\hline
  & {1}/{2}   &  {1}/{2} 
\end{array}
\end{equation}
respectively. Setting $z=\exp(i\zeta h)$, each of these methods yield a 
transfer matrix relation given by
\begin{equation}\label{eq:TM-relation-RK2}
{\vv{v}}_{n+1}=\frac{z^{-1}}{\Delta_{n+1}}M_{n+1}(z^2)\vv{v}_n,
\end{equation}
where 
\begin{equation}\label{eq:two-stage-TM}
M_{n+1}(z^2)=
\begin{pmatrix}
1+\frac{1}{4}Q_{n+1}R_{n}z^2 & \frac{1}{2}Q_{n}+\frac{1}{2}Q_{n+1}z^2\\
\frac{1}{2}R_{n}z^2+\frac{1}{2}R_{n+1} & z^2+\frac{1}{4}Q_{n}R_{n+1}
\end{pmatrix},
\end{equation}
while $\Delta_{n+1}=1-\frac{1}{4}Q_{n+1}R_{n+1}$ for the Lobatto IIIA and 
$\Delta_{n+1}=1-\frac{1}{4}Q_{n}R_{n}$ for Lobatto IIIB method. 

\subsubsection{Three-stage methods}
The Butcher tableau for Kutta's third order method, which is an explicit method, reads as
\begin{equation}
\renewcommand\arraystretch{1.2}
\begin{array}
{c|rrr}
0   & 0 & 0 &0\\
1/2 & {1}/{2} &  0 & 0\\
1   & -1      &  2 & 0\\
\hline
  & {1}/{6}   &  {2}/{3} & 1/6
\end{array}
\end{equation}
Setting $z=\exp(i\zeta h/2)$, this method simplifies to
\begin{equation}\label{eq:Kutta3}
{\vv{v}}_{n+1}=z^{-2}M_{n+1}(z^2){\vv{v}}_{n},
\end{equation}
where entries of the transfer matrix are given by
\begin{align*}
M^{(n+1)}_{11}(z^2) & = 1 
+ \frac{1}{3}\left(Q_{n+1/2}R_n + Q_{n+1}R_{n+1/2}\right)z^2-\frac{1}{6}Q_{n+1}R_nz^4,\\
M^{(n+1)}_{22}(z^2) &=z^4
+\frac{1}{3}\left(Q_nR_{n+1/2} + Q_{n+1/2}R_{n+1}\right)z^2-\frac{1}{6}Q_nR_{n+1},\\
M^{(n+1)}_{12}(z^2) &= \frac{1}{6}Q_n+\frac{1}{6}Q_{n+1}z^4
+\left(\frac{2}{3}Q_{n+1/2}+\frac{1}{6}Q_nQ_{n+1}R_{n+1/2}\right)z^2,\\
M^{(n+1)}_{21}(z^2) &= \frac{1}{6}R_{n+1}+\frac{1}{6}R_nz^4
+\left(\frac{2}{3}R_{n+1/2}+\frac{1}{6}Q_{n+1/2}R_nR_{n+1}\right)z^2.
\end{align*}
For testing purposes, we label this method as ``RK3''.

The next two examples are implicit methods based on $3$-point Lobatto
quadrature: The Butcher tableau for the Lobatto~IIIA method of order $4$ reads as
\begin{equation}
\renewcommand\arraystretch{1.2}
\begin{array}
{c|rrr}
0   & 0 & 0 &0\\
1/2 & {5}/{24} & {1}/{3} &-1/24\\
1   & {1}/{6} &  {2}/{3} &1/6\\
\hline
  & {1}/{6}   &  {2}/{3} & 1/6
\end{array}
\end{equation}
Again, setting $z=\exp(i\zeta h/2)$, the method simplifies to
\begin{equation}\label{eq:LobattoIIIA4}
{\vv{v}}_{n+1}=\frac{z^{-2}}{\Delta_{n+1}(z^2)}M_{n+1}(z^2){\vv{v}}_{n},
\end{equation}
where
\begin{multline}\label{eq:LIIIA3-TM}
M_{n+1}(z^2)
=\begin{pmatrix}
1+\frac{1}{12}z^2Q_{n+1}R_{n+1/2} &\frac{1}{6}z^2Q_{n+1}+\frac{1}{3}Q_{n+1/2}\\
\frac{1}{6}R_{n+1}+\frac{1}{3}z^2R_{n+1/2}& z^2+\frac{1}{12}R_{n+1}Q_{n+1/2}
\end{pmatrix}\times\\
\begin{pmatrix}
1+\frac{1}{12}z^2R_nQ_{n+1/2}&\frac{1}{6}Q_{n}+\frac{1}{3}z^2Q_{n+1/2}\\
\frac{1}{6}z^2R_{n}+\frac{1}{3}R_{n+1/2}&z^2+\frac{1}{12}Q_nR_{n+1/2}
\end{pmatrix},
\end{multline}
and
\begin{multline}
\Delta_{n+1}(z^2)
=\left(1+\frac{z^{-2}}{12}R_{n+1}Q_{n+1/2}\right)
\left(1+\frac{z^2}{12}Q_{n+1}R_{n+1/2}\right)\\
-\frac{1}{36}\left(Q_{n+1}+2{z^{-2}}Q_{n+1/2}\right)
\left(R_{n+1}+2{z^2}R_{n+1/2}\right).
\end{multline}
The Lobatto IIIB method of order 4 reads as
\begin{equation}
\renewcommand\arraystretch{1.2}
\begin{array}
{c|rrr}
0   & {1}/{6} & -{1}/{6} &0\\
1/2 & {1}/{6} &  {1}/{3} &0\\
1   & {1}/{6} &  {5}/{6} &0\\
\hline
  & {1}/{6}   &  {2}/{3} & 1/6
\end{array}
\end{equation}
This method simplifies to a form that is identical to~\eqref{eq:LobattoIIIA4} with
the transfer matrix given by~\eqref{eq:LIIIA3-TM} and 
\begin{multline}
\Delta_{n+1}(z^2)
=\left(1+\frac{z^{-2}}{12}Q_{n}R_{n+1/2}\right)
\left(1+\frac{z^2}{12}Q_{n+1/2}R_{n}\right)\\
-\frac{1}{36}\left(Q_{n}+2{z^2}Q_{n+1/2}\right)
\left(R_{n}+2z^{-2}R_{n+1/2}\right).
\end{multline}

\subsubsection{Four stage methods and above}
The fourth order classical RK method is given by
\begin{equation}
\renewcommand\arraystretch{1.2}
\begin{array}
{c|rrrr}
0   & 0       &  0      & 0 & 0\\
1/2 & {1}/{2} &  0      & 0 & 0\\
1/2 & 0       & {1}/{2} & 0 & 0\\
1   & 0       &  0      & 1 & 0\\
\hline
    & {1}/{6} & {1}/{3} & {1}/{3} & 1/6
\end{array}
\end{equation}
Setting $z=\exp(i\zeta h/2)$, this method simplifies to
\begin{equation}\label{eq:RK4}
{\vv{v}}_{n+1}=z^{-2}M_{n+1}(z^2){\vv{v}}_{n},
\end{equation}
where the entries of the transfer matrix are given by
\begin{equation*}
\begin{split}
M^{(n+1)}_{11}(z^2)&= 1+\frac{1}{6}Q_{n+1/2}R_{n+1/2}
+\frac{1}{6}\left(Q_{n+1/2}R_{n}+Q_{n+1}R_{n+1/2}\right)z^2\\
&\qquad+\frac{1}{24}Q_{n+1/2}Q_{n+1}R_nR_{n+1/2}z^4,\\
M^{(n+1)}_{22}(z^2)&=\frac{1}{24}Q_{n}Q_{n+1/2}R_{n+1/2}R_{n+1}
+\frac{1}{6}\left(Q_{n}R_{n+1/2} + Q_{n+1/2}R_{n+1}\right)z^2\\
&\qquad+\left(1+\frac{1}{6}Q_{n+1/2}R_{n+1/2}\right)z^4,
\end{split}
\end{equation*}
and,
\begin{align*}
M^{(n+1)}_{12}(z^2)
&=\frac{1}{6}Q_{n}\Xi_{n+1/2}+\frac{2}{3}Q_{n+1/2}z^2
+\frac{1}{6}Q_{n+1}\Xi_{n+1/2}z^4,\\
M^{(n+1)}_{21}(z^2)&=\frac{1}{6}R_{n}\Xi_{n+1/2}z^4+\frac{2}{3}R_{n+1/2}z^2
+\frac{1}{6}R_{n+1}\Xi_{n+1/2},
\end{align*}
where 
\[
\Xi_{n+1/2} = \left(1+\frac{1}{2}Q_{n+1/2}R_{n+1/2}\right).
\]

For the sake of completeness, let us discuss a general method to design
implicit RK methods with arbitrary number of stages: \emph{the collocation
method}~\cite{HNW1993}. It is 
a straight-forward recipe to design implicit RK methods which leads to dense $A$
matrix. Taking the nodes to be $c_k = (k-1)/(s-1)$, the coefficients of the RK
method are given by
\begin{equation}
a_{jk}=\int_0^{c_j}L_k(\xi)d\xi,\quad
b_k=\int_0^{1}L_k(\xi)d\xi=a_{sk},
\end{equation}
where
\begin{equation}
L_k(\xi)=\prod_{n\neq k,n=1}^s\frac{\xi-c_n}{c_k-c_n}.
\end{equation}
Clearly $a_{1k}\equiv0$. The order of convergence of these methods turn out to
be $\bigO{h^s}$ if $s$ is even while $\bigO{h^{s+1}}$ if $s$ is odd. 
An example of a $5$-stage method of order $6$ is
\begin{equation}
\renewcommand\arraystretch{1.5}
\begin{array}
{c|rrrrr}
0           &                0 &                0 &               0&               0&              0\\
\frac{1}{4} & \frac{251}{2880} & \frac{323}{1440} & -\frac{11}{120}& \frac{53}{1440}& -\frac{19}{2880}\\
\frac{1}{2} &   \frac{29}{360} &    \frac{31}{90} &    \frac{1}{15}&    \frac{1}{90}&  -\frac{1}{360}\\
\frac{3}{4} &   \frac{27}{320} &   \frac{51}{160} &    \frac{9}{40}&  \frac{21}{160}&  -\frac{3}{320}\\
1           &     \frac{7}{90} &    \frac{16}{45} &    \frac{2}{15}&   \frac{16}{45}&    \frac{7}{90}\\
\hline
    & \frac{7}{90} & \frac{16}{45} & \frac{2}{15} & \frac{16}{45} & \frac{7}{90}
\end{array}
\end{equation}
Note that on account of the dense nature of $A$, the computed transfer matrix
(for increasing number of stages) becomes increasingly complicated and 
cumbersome to implement. A remedy to this situation
is to design diagonally implicit RK methods on the uniform grid. It appears that
such a program would be more realistically pursued in the framework of 
\emph{general linear methods} introduced by Butcher and his
school~\cite{Butcher2003,J2009}. This method would be taken up in a forthcoming 
paper of this series.

We conclude this section with the example of a $6$-stage explicit method of order $5$ due
to Kutta~\cite{Butcher2003}:
\begin{equation}
\renewcommand\arraystretch{1.5}
\begin{array}
{c|rrrrrr}
0           &               &  &  &  & & \\
\frac{1}{4} &   \frac{1}{4} &  &  &  & & \\
\frac{1}{4} &   \frac{1}{8} &  \frac{1}{8}&  &  & & \\
\frac{1}{2} &             0 &            0& \frac{1}{2}&  & & \\
\frac{3}{4} &  \frac{3}{16} & -\frac{3}{8}& \frac{3}{8}& \frac{9}{16}& & \\
1           & - \frac{3}{7} & \frac{8}{7} & \frac{6}{7}&-\frac{12}{7}&\frac{8}{7}& \\
\hline
    & \frac{7}{90} &0& \frac{16}{45} & \frac{2}{15} & \frac{16}{45} & \frac{7}{90}
\end{array}
\end{equation}
We label this method as ``RK6''. The transfer matrix for this method can be
computed easily using any computer algebra system (the expressions are being
omitted for the sake of brevity of presentation).


\subsection{Linear multistep methods}\label{sec:LMM}
Fast nonlinear Fourier transform based on exponential linear multistep methods
(LMM) were introduced in~\cite{V2018LPT} and used as benchmarking tool
in~\cite{V2018BL}. This family of methods are well suited for vanishing boundary
condition; however, periodic boundary conditions would require a Runge-Kutta
based starting procedure. For the sake of comparison with the Runge-Kutta
approach, we revisit the discretization based on a general zero-stable LMM 
which can be stated as
\begin{equation}\label{eq:LMMs}
\sum_{s=0}^m\alpha_s\wtilde{\vv{v}}_{n+s} =
h\sum_{s=0}^m\beta_s\wtilde{U}_{n+s}\wtilde{\vv{v}}_{n+s},
\end{equation}
where $\vs{\alpha} = (\alpha_0,\alpha_1,\ldots,\alpha_m)$ and 
$\vs{\beta} = (\beta_0,\beta_1,\ldots,\beta_m)$. Without loss of generality, we
may set $\alpha_m\equiv1$. Solving for $\wtilde{\vv{v}}_{n+m}$, we have
\begin{equation}
\wtilde{\vv{v}}_{n+m}=
-\left(\sigma_0-\frac{\beta_m}{\alpha_m}h\wtilde{U}_{n+m}\right)^{-1}
\sum_{s=0}^{m-1}
\left({\alpha_s}\sigma_0-{\beta_s}h\wtilde{U}_{n+s}\right)\wtilde{\vv{v}}_{n+s},
\end{equation}
or, equivalently,
\begin{equation}
{\vv{v}}_{n+m}=-\left(\sigma_0-\beta_m h{U}_{n+m}\right)^{-1}
\sum_{s=0}^{m-1}e^{-i\sigma_3\zeta h(m-s)}
\left({\alpha_s}\sigma_0-\beta_s h{U}_{n+s}\right)
\vv{v}_{n+s}.
\end{equation}
Let $z=e^{i\zeta h}$ and 
\begin{equation}
\left(\sigma_0-\beta_mh{U}_{n+m}\right)^{-1}
e^{-i\sigma_3\zeta h(m-s)}
\left({\alpha_s}\sigma_0-\beta_s h{U}_{n+s}\right)
=\frac{z^{-(m-s)}}{\Theta_{n+m}}M^{(m-s)}_{n+m}(z^2),
\end{equation}
where $\Theta_{n+m}=(1-{\beta}_m^2R_{n+m}Q_{n+m})$ and 
\begin{equation}
M^{(m-s)}_{n+m}(z^2)=
\begin{pmatrix}
{\alpha}_s - {\beta}_m{\beta}_sQ_{n+m}R_{n+s}z^{2(m-s)}
&{\beta}_m{\alpha}_sR_{n+m}-{\beta}_sR_{n+s}z^{2(m-s)}\\
-{\beta}_sQ_{n+s} + {\beta}_m{\alpha}_sQ_{n+m}z^{2(m-s)}
&-{\beta}_m{\beta}_sQ_{n+s}R_{n+m}+{\alpha}_sz^{2(m-s)}
\end{pmatrix}.
\end{equation}
On simplifying, we have
\begin{equation}
z^m{\vv{v}}_{n+m}=-\frac{1}{\Theta_{n+m}}\sum_{s=0}^{m-1}M^{(m-s)}_{n+m}(z^2)z^{s}\vv{v}_{n+s},
\end{equation}
which can be put into a one-step form by defining
\begin{equation}
{\mathcal{M}}_{n+m}(z^2) =
\begin{pmatrix}
-M^{(1)}_{n+m}& -M^{(2)}_{n+m}&
\ldots& -M^{(m-1)}_{n+m}&-M^{(m)}_{n+m}\\
\Theta_{n+m}\sigma_0 &0&\ldots&0&0\\
0&\Theta_{n+m}\sigma_0&\ldots&0&0\\
\vdots&\vdots &\ddots & \vdots& \vdots\\
0&0&\ldots&\Theta_{n+m}\sigma_0&0
\end{pmatrix},
\end{equation}
so that
\begin{equation}
\pmb{\mathcal{W}}_{n+m}=\frac{1}{\Theta_{n+m}}{\mathcal{M}}_{n+m}(z^2)\pmb{\mathcal{W}}_{n+m-1},
\end{equation}
where $\vv{w}_n=z^{n}\vv{v}_n$ and 
\[
\pmb{\mathcal{W}}_{n}=(\vv{w}_{n},\vv{w}_{n-1},\ldots,\vv{w}_{n-m+1})^{\tp}\in\field{C}^{2m}.
\]
Let us consider the Jost solution $\vs{\phi}(t;\zeta)$. Set the computational
domain to be $\Omega=[T_1, T_2]$ with $-h\ell_-=T_1$ and $h\ell_+=T_2$. We
assume that $q(t)$ is supported in $\Omega$ and $q_n=0$ for 
$n=-m+1, -m+2,\ldots,0$ so that 
$\vs{\phi}_n=z^{\ell_-}z^{-n}(1,0)^{\tp}$ for $n=-m+1,-m+2,\ldots,0$. In order 
to express the discrete approximation to the Jost solutions, let us define 
the vector-valued polynomial
\begin{equation}\label{eq:poly-vec}
\vv{P}_n(z^2)
=\begin{pmatrix}
P^{(n)}_{1}(z^2)\\
P^{(n)}_{2}(z^2)
\end{pmatrix}
=\sum_{j=0}^{n}
\vv{P}^{(n)}_{j}z^{2j}
=\sum_{j=0}^{n}
\begin{pmatrix}
P^{(n)}_{1,j}\\
P^{(n)}_{2,j}
\end{pmatrix}z^{2j},
\end{equation}
such that $\vs{\phi}_n = z^{\ell_-}z^{-n}\vv{P}_n(z^2)$. The initial condition
works out to be 
\begin{equation}
\pmb{\mathcal{W}}_{0}
=z^{\ell_-}
\begin{pmatrix}
\vs{\phi}_{0}\\
z\vs{\phi}_{-1}\\
\vdots\\
z^{-m+1}\vs{\phi}_{-m+1}
\end{pmatrix}
=z^{\ell_-}
\begin{pmatrix}
\vs{P}_{0}(z^2)\\
\vs{P}_{-1}(z^2)\\
\vdots\\
\vs{P}_{-m+1}(z^2)
\end{pmatrix}
=z^{\ell_-}
\begin{pmatrix}
1\\
0\\
\vdots\\
1\\
0
\end{pmatrix},
\end{equation}
yielding the recurrence relation
\begin{equation}
\pmb{\mathcal{P}}_{n+m}(z^2)=\frac{1}{\Theta_{n+m}}{\mathcal{M}}_{n+m}(z^2)\pmb{\mathcal{P}}_{n+m-1}(z^2),
\end{equation}
where 
$\pmb{\mathcal{P}}_{n}(z^2) =
(\vv{P}_{n}(z^2),\vv{P}_{n-1}(z^2),\ldots,\vv{P}_{n-m+1}(z^2))^{\tp}\in\field{C}^{2m}$.
From~\eqref{eq:phi-init-final}, the discrete scattering coefficients work out to be
\begin{equation}
\begin{split}
&a_{N}(z^2)={P}^{(N)}_1(z^2),\\
&b_{N}(z^2)=(z^2)^{-\ell_{+}}{P}^{(N)}_2(z^2),
\end{split}
\end{equation}
which are uniquely defined for $\Re\zeta\in[-{\pi}/{2h},\,{\pi}/{2h}]$. 
Finally, the overall computational complexity of the direct NFT excluding the cost of
computing eigenvalues again works out to be
$\bigO{N(K+\log^2N)}$~\cite{V2018LPT}.

\subsubsection{Stability and convergence analysis}
The stability and convergence analysis of the multistep method can be carried
out by first formulating it as a one-step method. The discussion provided below
is taken from~\cite[Chap.~4, Part~III]{HNW1993}: Define the augmented vector
\begin{equation}
\wtilde{\vs{\Upsilon}}_n
=\left(\tilde{\vv{v}}_{n+m-1},\tilde{\vv{v}}_{n+m-2},\ldots,\tilde{\vv{v}}_{n}\right)^{\tp}\in\field{C}^{2s},
\end{equation}
and the diagonal matrix
\begin{equation}
\wtilde{D}_n=\diag\left(h\wtilde{U}_{n+m-1},\ldots,h\wtilde{U}_{n}\right)\in\field{C}^{2s\times2s}.
\end{equation}
Then the method~\eqref{eq:LMMs} can be stated as
\begin{multline}\label{eq:LMM-one-step}
\wtilde{\vs{\Upsilon}}_{n+1}
=\diag\left[\frac{1}{\Theta_{n+m}}\left(\sigma_0+{\beta}_mh\wtilde{U}_{n+m}\right),
\sigma_0,\ldots,\sigma_0\right]\times\\
\left[(A\otimes\sigma_0)+(B\otimes\sigma_0)\wtilde{D}_n\right]
\wtilde{\vs{\Upsilon}}_n,
\end{multline}
where
\begin{equation}
A=
\begin{pmatrix}
-{\alpha}_0&-{\alpha}_1&
\ldots&-{\alpha}_{m-2}&-{\alpha}_{m-1}\\
1 &0&\ldots&0&0\\
0&1&\ldots&0&0\\
\vdots&\vdots &\ddots & \vdots& \vdots\\
0&0&\ldots&1&0
\end{pmatrix},\quad
B=
\begin{pmatrix}
{\beta}_0&{\beta}_1&
\ldots&{\beta}_{m-1}\\
0 &0&\ldots&0\\
0&0&\ldots&0\\
\vdots&\vdots &\ddots & \vdots\\
0&0&\ldots&0
\end{pmatrix}.
\end{equation}
An LMM is stable if the generating polynomial 
\begin{equation}
\varrho(z)=\alpha_m z^m+\alpha_{m-1}z^{m-1}+\ldots+\alpha_0,
\end{equation}
satisfies the root condition, i.e., roots of $\varrho(z)$ lie on or within the
unit circle, and, the roots on the unit circle are simple. This is the 
well-known \emph{zero-stability} criteria. This property directly implies that 
$\|A\otimes\sigma_0\|\leq1$ where $\|\cdot\|$ is some matrix norm induced 
by a vector norm. Putting
\begin{multline}
\wtilde{\Lambda}_{n+1}(t_n;h)
=\frac{1}{h}\diag\left[\frac{1}{\Theta_{n+m}}\left(\sigma_0+{\beta}_mh\wtilde{U}_{n+m}\right),
\sigma_0,\ldots,\sigma_0\right]\times\\
\left[(A\otimes\sigma_0)+(B\otimes\sigma_0)\wtilde{D}_n\right]-\frac{1}{h}(A\otimes\sigma_0),
\end{multline}
it is easy to conclude that the method is \emph{consistent} since
\[
\wtilde{\Lambda}_{n+1}(t_n;0)=(B\otimes\sigma_0)(I_{2m}\otimes\wtilde{U}_{n})
=(B\otimes\wtilde{U}_{n}).
\]
The method~\eqref{eq:LMM-one-step} now written as
\begin{equation}
\wtilde{\vs{\Upsilon}}_{n+1}=
(A\otimes\sigma_0)\wtilde{\vs{\Upsilon}}_n+
h\wtilde{\Lambda}_{n+1}(t_n;h)\wtilde{\vs{\Upsilon}}_n,
\end{equation}
makes it manifestly evident that it is convergent provided 
$\|\wtilde{\Lambda}_{n+1}(t_n;h)\|_s$ is bounded which follows from the
boundedness of the potential.

\subsubsection{Explicit Adams Method}
The explicit Adams method (see Table~\ref{tb:EA}) can be stated as
\begin{equation}\label{eq:explicit-Adams}
\wtilde{\vv{v}}_{n+m}-\wtilde{\vv{v}}_{n+m-1}=
h\sum_{s=0}^{m-1}\beta_s\wtilde{U}_{n+s}\wtilde{\vv{v}}_{n+s},
\end{equation}
where $\beta_m$ is identically $0$ so that $\Theta_{n+m}=1$ and
\begin{equation}
M^{(m-s)}_{n+m}(z^2)=
\begin{pmatrix}
{\alpha}_s&-{\beta}_sQ_{n+s}\\
-{\beta}_sR_{n+s}z^{2(m-s)}&{\alpha}_sz^{2(m-s)}
\end{pmatrix},\quad s<m.
\end{equation}

\begin{table}[tbph!]
\def\arraystretch{1.75}
\caption{Explicit Adams Method\label{tb:EA}}
\begin{center}
\begin{tabular}{ccc}
Method & $\vs{\beta}$ & Order\\
\hline
EA$_1$ & $(1,0)$ & $1$\\ 
\hline
EA$_2$ & $(-\frac{1}{2},\frac{3}{2},0)$ & $2$\\ 
\hline
EA$_3$ & $(\frac{5}{12},-\frac{16}{12},\frac{23}{12},0)$ &$3$\\ 
\hline
EA$_4$ & $(-\frac{9}{24},\frac{37}{24},-\frac{59}{24},\frac{55}{24},0)$ &$4$\\ 
\hline
EA$_5$ &
$(\frac{251}{720},-\frac{1274}{720},-\frac{2616}{720},-\frac{2774}{720},\frac{1901}{720},0)$ &$5$\\ 
\hline\\
\end{tabular}
\end{center}
\end{table}

\subsubsection{Implicit Adams Method}
The implicit Adams method (see Table~\ref{tb:IA}) can be stated as
\begin{equation}\label{eq:implicit-Adams}
\wtilde{\vv{v}}_{n+m}-\wtilde{\vv{v}}_{n+m-1}=
h\sum_{s=0}^{m}\beta_s\wtilde{U}_{n+s}\wtilde{\vv{v}}_{n+s},
\end{equation}
so that $\Theta_{n+m}=(1-{\beta}_m^2R_{n+m}Q_{n+m})$ with
\begin{equation}
M^{(1)}_{n+m}(z^2)
=-\begin{pmatrix}
1 + {\beta}_m{\beta}_{m-1}Q_{n+m}R_{n+m-1}z^{2}
&{\beta}_sQ_{n+m-1} + {\beta}_mQ_{n+m}z^{2}\\
{\beta}_mR_{n+m}+{\beta}_{m-1}R_{n+s}z^{2}
&{\beta}_m{\beta}_{m-1}Q_{n+s}R_{n+m}+z^{2}
\end{pmatrix},
\end{equation}
and, for $s<m-1$,
\begin{equation}
M^{(m-s)}_{n+m}(z^2)
=-{\beta}_s
\begin{pmatrix}
{\beta}_mQ_{n+m}R_{n+s}z^{2(m-s)}
&Q_{n+s}\\
R_{n+s}z^{2(m-s)}
&{\beta}_m Q_{n+s}R_{n+m}
\end{pmatrix}.
\end{equation}

\begin{table}[tbph!]
\def\arraystretch{1.75}
\caption{Implicit Adams Method\label{tb:IA}}
\begin{center}
\begin{tabular}{ccc}
Method & $\vs{\beta}$ & Order\\
\hline
IA$_1$ & $(\frac{1}{2},\frac{1}{2})$ & $2$\\ 
\hline
IA$_2$ & $(-\frac{1}{12},\frac{8}{12},\frac{5}{12})$ & $3$\\ 
\hline
IA$_3$ & $(\frac{1}{24},-\frac{5}{24},\frac{19}{24},\frac{9}{24})$ &$4$\\ 
\hline
IA$_4$ &
$(-\frac{19}{720},\frac{106}{720},-\frac{264}{720},\frac{646}{720},\frac{251}{720})$ &$5$\\ 
\hline\\
\end{tabular}
\end{center}
\end{table}
\subsubsection{Backward Differentiation Formula}
Discretization using the BDF (see Table~\ref{tb:BDF}) can be stated as 
\begin{equation}
\sum_{s=0}^m\alpha_s\wtilde{\vv{v}}_{n+s} =
h\beta\wtilde{U}_{n+m}\tilde{\vv{v}}_{n+m},
\end{equation}
where $\beta_m\equiv\beta$ so that $\Theta_{n+m}=(1-{\beta}^2R_{n+m}Q_{n+m})$ and
\begin{equation}
M^{(m-s)}_{n+m}(z^2)=
{\alpha}_s\begin{pmatrix}
1      & {\beta}Q_{n+m}z^{2(m-s)},\\
{\beta}R_{n+m} & z^{2(m-s)}
\end{pmatrix},\quad s<m.
\end{equation}

\begin{table}[tb!]
\def\arraystretch{1.75}
\caption{\label{tb:BDF}Backward differentiation formula}
\begin{center}
\begin{tabular}{ccc }
{Method} & $\vs{\alpha}$ & $\beta$\\
\hline
{BDF}$_1$ & $\left(-1,1\right)$                        & $1$\\
\hline
{BDF}$_2$ & $\left(\frac{1}{3},-\frac{4}{3}, 1\right)$ & $\frac{2}{3}$\\
\hline
{BDF}$_3$ & $\left(-\frac{2}{11}, \frac{9}{11}, -\frac{18}{11},1\right)$ &
$\frac{6}{11}$\\
\hline
{BDF}$_4$ & $\left(\frac{3}{25}, -\frac{16}{25},
\frac{36}{25},-\frac{48}{25},1\right)$ & $\frac{12}{25}$\\
\hline
{BDF}$_5$ & $\left(-\frac{12}{137},\frac{75}{137},-\frac{200}{137},
\frac{300}{137},-\frac{300}{137},1\right)$ & $\frac{60}{137}$\\
\hline
{BDF}$_6$ & $\left(\frac{10}{147}, -\frac{72}{147},\frac{225}{147},
-\frac{400}{147},\frac{450}{147},-\frac{360}{147},1\right)$ & $\frac{60}{147}$\\
\hline
\end{tabular}
\end{center}
\end{table}

\section{Conclusion}
To conclude, we explored in this paper the theoretical aspects of the
exponential Runge-Kutta and linear multistep methods for the numerical solution
of the Zakharov-Shabat scattering problem. The discrete system obtained can be 
conveniently stated using a transfer matrix (TM) relation whose explicit form was 
provided for both of these family of methods. Looking at the structure of the 
TMs involved, it can be concluded that the Runge-Kutta methods lead to a more 
optimal TM relation as compared to that of the linear multistep methods; however, 
the TM for the linear multistep methods
is far more easy to derive as well as implement. Both of these family of methods
lead to a discrete system that is amenable to FFT-based fast polynomial
arithmetic which yields fast algorithms for the computation of discrete
scattering coefficients. 

Based on the ideas presented in this paper, it is not difficult conclude that
(exponential) general linear methods developed by 
Butcher~\cite{Butcher2003,J2009} can also be exploited to obtain a family of fast 
NFT algorithms. These ideas will be explored in
forthcoming paper where we hope to present a comprehensive comparative study of
these numerical algorithms.
\appendix
\section{Some useful identities}
Cosider a $2\times2$ matrix $A$, then
\begin{equation}
A=\begin{pmatrix}
a_{11} & a_{12}\\
a_{21} & a_{22}
\end{pmatrix},\quad
A^{-1}=\frac{1}{\det{A}}
\begin{pmatrix}
 a_{22} & -a_{12}\\
-a_{21} &  a_{11}
\end{pmatrix}.
\end{equation}
Further,
\begin{equation}
\sigma_1A\sigma_1=
\begin{pmatrix}
a_{22} & a_{21}\\
a_{12} &  a_{11}
\end{pmatrix},
\quad\sigma_2A\sigma_2=
\begin{pmatrix}
a_{22} & -a_{21}\\
-a_{12} &  a_{11}
\end{pmatrix},
\end{equation}
and,
\begin{align}
&e^{i\sigma_3\lambda}Ae^{-i\sigma_3\lambda}
=\begin{pmatrix}
a_{11}               & a_{12}e^{2i\lambda}\\
a_{21}e^{-2i\lambda} & a_{22}
\end{pmatrix},\\
&\left(e^{i\sigma_3\lambda}Ae^{-i\sigma_3\lambda}\right)^{-1}=
e^{i\sigma_3\lambda}A^{-1}e^{-i\sigma_3\lambda}.
\end{align}

If $A$, $B$, $C$ and $D$ are matrices of such size that one can form the matrix
products $AC$ and $BD$, then
\[
(A\otimes B)(C\otimes D)=(AC)\otimes(BD).
\]

Suppose $A$, $B$, $C$, and $D$ are matrices of dimension $n\times n$, $n\times
m$, $m\times n$, and $m\times m$, respectively. If $A$ is invertible, one has
\[
\det\begin{pmatrix}
A & B\\
C & D \end{pmatrix}
=\det(A)\det(D - CA^{-1}B),
\] 
while, if $D$ is invertible, we have
\[
\det\begin{pmatrix}
A & B\\
C & D \end{pmatrix}
=\det(D)\det(A - BD^{-1}C).
\] 
Two interesting special cases when $m=1$ are as follows:
\[
\det\begin{pmatrix}
A & B\\
C & 1 \end{pmatrix}
=\det(A - BC),
\] 
and
\[
\det\begin{pmatrix}
A & B\\
C & 0 \end{pmatrix}
=\det(A - BC)-\det A.
\] 


\providecommand{\noopsort}[1]{}\providecommand{\singleletter}[1]{#1}%

\end{document}